\documentclass[11pt]{amsart}

\usepackage{graphs}
\usepackage{amssymb}
 \usepackage{color}
\usepackage{amscd}
\newtheorem{thm}{Theorem}[section]
\newtheorem{prop}[thm]{Proposition}
\newtheorem{lem}[thm]{Lemma}
\newtheorem{cor}[thm]{Corollary}

\theoremstyle{definition}
\newtheorem{definition}[thm]{Definition}
\theoremstyle{remark}
\newtheorem{remark}[thm]{Remark}
\newtheorem{example}[thm]{Example}

\numberwithin{equation}{section}

\newcommand{\wt}{\widetilde}

\begin{document}


\title[Cuntz-Krieger relations and Bratteli diagrams]{Representations of Cuntz-Krieger relations, dynamics on Bratteli diagrams, and path-space measures}


\author{Sergey Bezuglyi}
\address{Department of Mathematics, Institute for Low Temperature Physics, Kharkiv, Ukraine.
\textit{Current address: Department of Mathematics, University of Iowa,
Iowa City, IA 52442}}
\email{bezuglyi@gmail.com}


\author{Palle E.T. Jorgensen}
\address{Department of Mathematics,  University of Iowa,
Iowa City, IA 52442}
\email{palle-jorgensen@uiowa.edu}



\begin{abstract}
We study a new class of representations of the Cuntz-Krieger algebras $\mathcal O_A$ constructed by semibranching function systems, naturally related to stationary Bratteli diagrams. The notion of isomorphic semibranching function systems is defined and studied. We show that any isomorphism of such systems  implies  the equivalence of the corresponding representations of Cuntz-Krieger algebra $\mathcal O_A$. In particular, we show that equivalent measures generate equivalent representations of $\mathcal O_A$. We use Markov measures which are defined  on the path space of stationary Bratteli  diagrams to construct isomorphic representations of $\mathcal O_A$.  To do this, we  associate a (strongly) directed graph to a stationary (simple) Bratteli diagram, and show that isomorphic  graphs generate isomorphic semibranching function systems. We also consider a class of monic representations of the Cuntz-Krieger algebras, and we classify them up to unitary equivalence. Several examples that illustrate the results are included in the paper.

\end{abstract}


\maketitle


\section{Introduction}\label{intro}

Our paper is at the cross-roads of non-commutative and commutative harmonic analysis. The non-commutative component deals with representations of a class of $C^*$-algebras, called the Cuntz-Krieger algebras, written $\mathcal O_A$, and indexed by a matrix $A$ ; it is a matrix over the integers, called the incidence matrix. As the matrix  $A$  varies, so does (i) the Bratteli diagram, (ii) the $C^*$-algebra, and (iii) the associated compact measure space.
      The representations of $\mathcal O_A$  that we study, and find, are motivated by, and have applications to, an important commutative problem; the problem of understanding ``dynamics on infinite path-space measures''. The latter dynamical systems are called Bratteli diagrams, although this may be a stretch.
      A Bratteli diagram is a certain discrete graph (see details in the paper), but there is a dual object which is a compact space, from which we get a family of measure-space dynamical systems. It is the latter we study here. We add that the afore mentioned duality, i.e., discrete vs compact, generalizes Pontryagin's duality for Abelian groups: The Pontryagin-dual of a discrete abelian group, is a compact Abelian group, and vice versa. But we study cases here which go beyond the category of groups.
     Both our non-commutative questions about representations of $\mathcal O_A$, and the associated Abelian questions about equivalence of measure spaces, are quite subtle. Our overall plan is to prove theorems on representations of the $\mathcal O_A$  $C^*$-algebras, and then to apply them to classifications of dynamics on this family of compact path-space measure dynamical systems.

\subsection{ Motivation and earlier papers}    Our main theme is the question of deciding equivalence of path-space dynamical systems on Bratteli diagrams  $X_B$.
We note that Bratteli diagrams in a variety of guises have found applications to such areas of analysis, both commutative, and non-commutative, as representations of groups, classification of approximately finite dimensional (AF) $C^*$-algebras, and algebraic $K$-theory. To give a sample, we mention the following papers \cite{BrJKR00}, \cite{BrJKR02a}, \cite{BrJKR02b}, \cite{BrJO2004},  \cite{Ell89}, \cite{EllMu93} and the works cited there. In these papers, all three versions of the diagrams are used: (1) graphs (vertices, edges), (2) ordered discrete abelian groups (representing $K_0$ and $K_1$ groups), and (3) compact abelian groups (their duals).

In the present paper, we study equivalence of pairs of measures  $\mu$   arising in this class of dynamical systems (symbolic dynamics). Our answers are presented in terms of the representations of certain Cuntz-Krieger (CK) algebras \cite{CuKr80} used in generating the particular system. The CK-algebras are $C^*$-algebras, but our focus is on their representations in Hilbert spaces of the form $L^2(X_B, \mu)$ where the measures  $\mu$ are as described.  Departing from more traditional dynamical systems, and stochastic processes, we study here systems derived from one-sided endomorphisms, and our measures are typically not Gaussian.

In our setting, we study measures on one-sided infinite paths  $X_B$;  i.e., paths represented by infinite words, and with the alphabet in turn represented by vertices in the given Bratteli diagram $B$. For a given infinite word, the choice of letter at place $n$  is from a set of vertices $V_n$  of $B$; so with the vertex set depending on the given Bratteli diagram. In the special case of the Cuntz algebra $\mathcal O_N$ ($N$ fixed), \cite{Cu77, Cu79}, we may use the same alphabet $\mathbb Z_N$ (the cyclic group of order $N$), at each place in an infinite word.  In this case $X_B$ is a Cantor group. The case of $\mathcal O_N$ has been studied in recent papers \cite{DJ14a, DJ14b}.  In both cases, Cuntz, and Cuntz-Krieger, by a representation we mean an assignment of isometries  in a Hilbert space  $H$, assigning to every letter from the finite alphabet an isometry, and assigned  in such a way that distinct isometries have orthogonal ranges, adding up to the identity operator in $H$, in the case of $\mathcal O_N$. By contrast, in the case of the Cuntz-Krieger algebras, the corresponding sum-relation depends on a choice of incidence matrix $A$, the matrix defining the CK-algebra under consideration, see (\ref{CK relations}) in subsection 2.5 below.

This particular family of representations is motivated in part by quantization of particle systems.  Now the �quantized� system must be realized as an $L^2$-space with respect to a suitable measure on some path-space.  Here we study the case of path-space dynamical systems $X_B$  on Bratteli diagrams $B$. Because of this particular setting, it is not be possible to use of more traditional Gaussian measures (on infinite product spaces); and as an alternative we suggest a family of quasi-stationary Markov measures. In each of our symbolic representations we use one-sided shift to the left; and a system of a finite number of branches (vertices) defining shift to the right, shifting to the right, and filling in a letter from the alphabet. As a result, we get Markov measures  on $X_B$, see e.g., \cite{Ak12, JP12}.

When associated representations of the Cuntz-Krieger algebras are brought to bear, we arrive at useful non-commutative versions of these (commutative) symbolic shift mappings. For earlier related work, see e.g., \cite{DPS14, DJ12, Sk97, GPS95}.

In more traditional instances of the dynamics problem in its commutative and non-commutative guise, see e.g.,  \cite{BrRo81, Hi80} , there are three typical methods of attacking the question of equivalence or singularity of two measures: (1) Kakutani's theorem for infinite-product measures \cite{Ka48}, which asserts that two infinite product measures are either equivalent or mutually singular, (2)  methods based  on entropy considerations, and (3) the method of using the theory of reproducing kernels; see \cite{Hi80}. Because of the nature of our setting, one-sided shifts, commutative and non-commutative, we must depart from the setting of Gaussian measures. As a result, in our present study of equivalence or singularity of two measures, vs equivalence (or disjointness) of representations, only ideas from (1) seem to be applicable here.

\subsection{Representations of the Cuntz-Krieger algebras} A Cuntz-Krieger algebra $\mathcal O_A$ is a $C^*$-algebra on generators and relations as in (\ref{CK relations}) and depending on a fixed 0-1 matrix $A$.
Recently there has been an increased interest in use of the Cuntz-Krieger algebras and their representations in dynamics (including the study of fractals, and geometric measure theory), in ergodic theory, and in quantization questions from physics. Perhaps this is not surprising since Cuntz, and Cuntz-Krieger algebras are infinite algebras on a finite number of generators, and defined from certain relations. By their nature, these representations reflect intrinsic self-similar inherent in the problem at hand; and thus they serve ideally to encode iterated function systems (IFSs), their dynamics, and their measures. The study of representations of operator algebras related to Bratteli diagrams have found increasing use in pure and applied problems, such as physics, wavelets, fractals, and signals.

Our paper is partitioned in five sections. In Section 2, we collect all necessary definitions and prove some auxiliary results that are used in the paper. More precisely, we first recall the concept of a Bratteli diagram, focusing mostly on the case of a stationary Bratteli diagram. The structure of the path space of a diagram allows considering an analogue of Markov measures. We show in detail how such measures can be constructed. Combinatorial properties of stationary (simple) Bratteli diagrams allow us to define (strongly) directed graphs (we call them coupled graphs) that are  naturally associated to  such  diagrams. It turns out that properties of stationary diagrams can be translated into properties of the corresponding graphs and vice versa. In this section, we also consider one of the main ingredients of our methods, this is the notion of semibranching function systems (s.f.s.) defined on a measure space $(X, \mu)$. It turns out that one can define at least two s.f.s. acting on the path space of a stationary Bratteli  diagram $B$, which are indexed by either vertices or edges of $B$. We find relations between these s.f.s.; they lead to a general concept of refinement of a s.f.s. The importance of s.f.s. follows from the result proved in \cite{MaPa11}: any such a system generates a representation of the Cuntz-Krieger algebras on the Hilbert space $L^2(X,\mathcal B, \mu)$.  This idea is elaborated in the present paper  where we study such representations of the Cuntz-Krieger algebras $\mathcal O_{\wt A}$ and $\mathcal O_A$. Here $\wt A$ is the matrix ``dual'' to $A$ whose non-zero entries are defined by linked edges of $B$ when $s(f) = r(e)$ (see the details below). In order to have a s.f.s., a measure on the path space $X_B$ must be defined.

Amongst the variety of possible measures, we work with quasi-stationary Markov measures and shifts on the path space of a diagram. In Section 3, we define the notion of isomorphic s.f.s. and prove that an isomorphism   of two s.f.s. implies the equivalence of the corresponding representations of the Cuntz-Krieger algebra $\mathcal O_A$. In particular, equivalent measures will give equivalent representations of $\mathcal O_A$.
Section 4 is devoted to the study of coupled graphs related to
stationary Bratteli diagrams. Isomorphic coupled graphs produce a map (called an admissible map) that, in its turn, generate an isomorphism of s.f.s. This allows us to find new necessary and sufficient conditions for the equivalence of representations of $\mathcal O_A$. In the last section we define and study a class of representations of $\mathcal O_A$ called {\em monic} representations. They are characterized by the property that the abelian subalgebra $\mathcal D_A$  of $\mathcal O_A$ has a cyclic vector. We prove that such representations can arise only when we take a monic system naturally related to the s.f.s defined on  the topological Markov chain $X_A$ (see Theorem \ref{thm on inherent systems}).

\section{Ingredients of the main results}\label{main ingredients}

In this section, we collect the definitions of the main notions and results that are used in the paper. The most important concepts defined below are (stationary) Bratteli diagrams, coupled directed graphs, and semibranching function systems.

\subsection{Stationary Bratteli diagrams}\label{Bratteli diagrams}

Here we  give the necessary definitions in the context of  general Bratteli diagrams. These definitions are utilized mostly in a particular case of stationary diagrams. We recall that Bratteli diagrams are used in Cantor dynamics to produce convenient models of homeomorphisms of a Cantor set \cite{GPS95}, \cite{HPS92}, \cite{Me06}. For instance, any substitution dynamical systems is represented as a homeomorphism (Vershik map) acting on the path space of a stationary Bratteli diagram \cite{F97}, \cite{DHS99}, \cite{BKM09}, \cite{BKMS10}. Because of the transparent structure of stationary diagrams, many important properties of such dynamical system can be explicitly computed.

\begin{definition}
A \textit{Bratteli diagram} is an infinite graph $B = (V,E)$ such that the vertex set $V = \bigcup_{i\geq 0}V_{i}$ and the edge set $E = \bigcup_{i\geq 1}E_{i}$ are partitioned into disjoint subsets $V_{i}$ and $E_{i}$ such that

(i) $V_{0} = \{v_{0}\}$ is a single point;

(ii) $V_{i}$ and $E_{i}$ are finite sets;

(iii) there exist a range map $r$ and a source map $s$ from $E$ to $V$ such that $r(E_{i}) = V_{i}$, $s(E_{i}) = V_{i-1}$, and $s^{-1}(v)\neq \emptyset$, $r^{-1}(v')\neq \emptyset$ for all $v \in V$ and $v' \in V \setminus V_{0}$.
\end{definition}

The pair $(V_{i}, E_{i})$ or just $V_i$ is called the $i$-th level of the diagram $B$.
A sequence (finite or infinite) of edges $(e_{i} : e_{i} \in E_{i})$ such that $r(e_{i}) = s(e_{i + 1})$ is called a \textit{path}. We denote by $X_{B}$ the set of all infinite paths starting at the vertex $v_{0}$. We suppose that $X_B$ is endowed  with the clopen topology generated by cylinder sets (finite paths, in other words) such that $X_{B}$ turns out a  Cantor set. This can be done for any simple Bratteli diagram and for a wide class of non-simple diagrams that do not have isolated points.

Given a Bratteli diagram $B = (V,E)$, define a sequence of incidence matrices  $F_{n} = (f_{v,w}^{(n)})$ of $B$ where $f_{v,w}^{(n)} = |\{e \in E_{n+1} : r(e) = v, s(e) = w\}|$ and  $v \in V_{n+1}$ and $w \in V_{n}$. Here and thereafter $|V|$ denotes the cardinality of the set $V$. The transpose matrix $F_n^T$ will be denoted by $A_n$.

A Bratteli diagram is called \textit{stationary} if $F_{n} = F_{1}=F$ for every $n \geq 2$. For a stationary diagram $B$ the notation $V$ and $E$ will stand for the sets of vertices of any level and the set of edges between any two consecutive levels below the first one. With some abuse of terminology,  we will also call $A = F^T$ the incidence matrix of the stationary diagram $B$.

A Bratteli diagram $B' = (V',E')$ is called the \textit{telescoping} of a Bratteli diagram $B = (V,E)$ to a sequence $0 = m_0 < m_1 < ...$ if $V_n' = V_{m_n}$ and $E_n'$ is the set of all paths from $V_{m_{n-1}}$ to $V_{m_n}$, i.e. $E_n' = E_{m_{n-1}}\circ...\circ E_{m_n} = \{(e_{m_{n-1}},...,e_{m_n}) : e_i \in E_i, r(e_i) = s(e_{i+1})\}$.

Observe that every vertex $v \in V$ is connected to $v_{0}$ by a finite path, and the set $E(v_0,v)$ of all such paths is finite. A Bratteli diagram is called \textit{simple} if for any $n > 0$ there exists $m > n$ such that any two vertices $v \in V_n$ and $w \in V_m$ are connected by a finite path. Using the telescoping procedure, we can always assume, without loss of generality, that any pair of vertices from two consecutive levels are connected by at least one edge. In case of a stationary Bratteli diagram $B$, the simplicity of $B$ is equivalent to the primitivity of the incidence matrix $F$.

\begin{definition}
Let $B = (V,E)$ be a Bratteli diagram. Two infinite paths $x = (x_{i})$ and $y = (y_{i})$ from $X_{B}$ are called \textit{tail equivalent} if there exists $i_{0}$ such that $x_{i} = y_{i}$ for all $i \geq i_{0}$. Denote by $\mathcal{R}$ the tail equivalence relation on $X_{B}$.
\end{definition}

It can be easily  seen that a Bratteli diagram is simple if and only if the tail equivalence relation $\mathcal R$ is \textit{minimal}; i.e., for arbitrary path $x \in X_B$ the set $\{y \in X_B : y \mbox{ is\ tail\ equivalent\ to\ } x\}$ is dense in $X_B$.

\begin{definition}
A {\em cylinder set} in the path space $X_B$ is the set $ \{x = (x_{i}) \in X_{B} : x_{i} = e_{i}, i = 1,...,n\}  =: X_{w}^{(n)}(\overline{e})$, where  $\overline{e} = (e_{1}, \ldots ,e_{n}) \in E(v_{0}, w)$, $n \geq 1$; we set
$$
X_{w}^{(n)} = \bigcup_{\overline e \in E(v_0, w)} X_{w}^{(n)}(\overline{e}).
$$
The cylinder set $X_{w}^{(n)}(\overline{e})$ will also be denoted as $[\overline e] =
[(e_0, e_1, ... , e_n)]$.
\end{definition}

Clearly, $\xi_n = \{X_{w}^{(n)} : w \in V_n\}$ forms a refining sequence of clopen  partition of $X_B$, $n\in \mathbb N$.

We recall the following facts that are widely used for the study of stationary Bratteli diagrams.

Let $\mathcal{A} = \{a_1, ... ,a_s\}$ be a finite alphabet, $\mathcal{A}^{*}$  the collection of finite non-empty words over $\mathcal{A}$. Denote by $\Omega = \mathcal{A}^\mathbb{Z}$, the set of all two-sided infinite sequences on $\mathcal{A}$. A {\em substitution} $\tau$ is a map $\tau \colon \mathcal{A} \rightarrow \mathcal{A}^{*}$. It extends to maps $\tau \colon \mathcal{A}^{*} \rightarrow \mathcal{A}^{*}$, and $\tau \colon \Omega \rightarrow \Omega$ by concatenation.
Denote by $T$ the shift on $\Omega$, that is $T(... x_{-1}.x_0 x_1 ...) = (... x_{-1}x_0.x_1x_2 ...)$.

For $x \in \Omega$, let $L_n(x)$ be the set of all words of length $n$ occurring in $x$; we set $L(x) = \bigcup_{n \in \mathbb{N}} L_n(x)$. The language of $\tau$ is the set $L_\tau$ of all finite words occurring in $\tau^n(a)$ for some $n \geq 0$ and $a \in \mathcal{A}$. The set $X_\tau ::= \{x \in \Omega : L(x) \subset L_\tau\}$ is $T$-invariant.
The dynamical system $(X_\tau, T_\tau)$, where $T_\tau$ is the restriction of $T$ to the $T$-invariant set $X_\tau$, is called \textit{the substitution dynamical system} associated to $\tau$.

Depending on properties of $\tau$, the system  $(X_\tau, T_\tau)$ can be minimal or, more generally, aperiodic.
As proved in \cite{DHS99} (for a minimal homeomorphism $T_\tau$) and \cite{BKM09} (for aperiodic $T_\tau$), there is a one-to-one correspondence between substitution dynamical systems and  stationary Bratteli diagrams where the dynamics is generated by the so called Vershik map.

The following result simplifies, in general, the study of stationary Bratteli diagrams.

\begin{lem} [\cite{DHS99}, \cite{GPS95}]\label{0-1_matrix}
Given a stationary Bratteli diagram $B$,  there exists a stationary Bratteli diagram $B'$ such that:
\begin{enumerate}
\item $|E(v_0, v)| =1, \forall v \in V$,

\item  the incidence matrix $F'$ is a 0-1 matrix,

\item  $B$ and $B'$ are isomorphic Bratteli diagrams.
\end{enumerate}
\end{lem}

A Bratteli diagram $B$ satisfying conditions (1)  and (2) of Lemma \ref{0-1_matrix} will be called (with some abuse of terminology) a \textit{0-1 Bratteli diagram}.

In fact, an obvious modification of Lemma \ref{0-1_matrix} remains true for arbitrary Bratteli diagram.

\begin{remark}
We observe that the path space of a stationary Bratteli diagram can be endowed with a group structure.

Let $F$ be a $d\times d$ matrix over $\mathbb Z$ with transposed $A$; then  $A$ (and $F$) acts on $\mathbb R^d$ by matrix-multiplication: $x \mapsto Ax,\ x \in \mathbb R^d$; and this action passes to the quotient $\mathbb T^d := \mathbb R^d/ \mathbb Z^d$, $x\ (\mathrm{mod}\ \mathbb Z^d) \mapsto Ax \ (\mathrm{mod}\ \mathbb Z^d)$. Setting $[x] := x\ \mod \mathbb Z^d$, we write  $A[x] := [Ax],\ x \in \mathbb T^d$.

The system of mappings
\begin{equation*}
D_F := \mathbb Z^d \stackrel{F}\longrightarrow \mathbb Z^d \stackrel{F}\longrightarrow  \cdots  \stackrel{F}\longrightarrow \mathbb Z^d  \stackrel{F}\longrightarrow \cdots,
\end{equation*}
defines the \textit{inductive limit group}  $D_F$. We recall that the discrete abelian group $D_F$  is formed by equivalence classes of elements $(i, x) \in \mathbb N_0 \times \mathbb Z^d$ with respect to the equivalence relation $((i,x) \sim (j, y) ) \Longleftrightarrow  (\exists n, m \in \mathbb N_0,\ (F)^nx = (F)^m y)$.
If $F$ is invertible, then $(F)^{-i} \mathbb Z^d \hookrightarrow (F)^{-i-1} \mathbb Z^d$. In this case, the equivalence relation $\sim$ is generated by $(i, x) \sim (i+1, Fx)$, so that
$$
D_F = \left(\bigcup_{k\in \mathbb Z}(F)^{-k}\right)\biggl/\mathbb Z^d \ \subset \ \mathbb Q^d.
$$

Using Pontryagin duality for locally compact abelian groups, we get the compact dual group $(D_F)^*$ which is realized as a {\em projective limit group}; it is also called a compact solenoid,
$$
Z_A := (D_F)^* = \{ (x_k) \in \prod_{k\in \mathbb N_0} \mathbb Z_d : A[x_{k+1}] = [x_k],\ \forall k \in \mathbb N_0\},
$$
i.e., $Ax_{k+1}\ (\mathrm{mod}\ \mathbb Z^d)   = x_k\  (\mathrm{mod}\ \mathbb Z^d)$. Here $\mathbb Z_d$ denotes the cyclic group of order $d$, and $\prod \mathbb Z_d$ is embedded into  $\prod \mathbb T^d$ via $A$.

\begin{lem}
As  a compact Cantor space, the path space $X_B$ of a stationary Bratteli diagram $B$ with incidence matrix $F$ is homeomorphic to the compact abelian group $Z_A = (D_F)^* $.
\end{lem}

\begin{proof} (Sketch)  As mentioned above, we can assume that  $B$ is a 0-1 Bratteli diagram. For the  0-1 incidence matrix $F$, there is an alphabet $\Sigma$ (see \cite{CuKr80}) such that
$$
X_B = \{(s_k) \in \prod_{k\in \mathbb N_0} V : f_{s_k, s_{k+1}} = 1\}.
$$
Picking a set of elements $V$ in the finite quotient $\mathbb Z^d/A\mathbb Z^d$ and using the Pontryagin duality $Z_A = (D_F)^* $, we see that there is a homeomorphism $X_B \longleftrightarrow Z_A$ via $A[x^{(k+1)}_{s_{k+1}}] = x_{s_k}^{(k)}$.
\end{proof}

It follows from the proved result that the path space $X_B$ has the structure of a compact abelian group $Z_A$ with the probability Haar measure.
\end{remark}

\subsection{Measures on Bratteli diagrams}\label{measures on BD}
Here we give a few definitions and facts related to a class of probability Borel measures on the path space $X_B$ of a Bratteli diagram $B$, stationary and non-stationary ones. More details can be found in \cite{BKMS10} and \cite{BJ14}. In order to avoid some unnecessary complications,  we make the following assumption: \textit{all measures considered in this paper are assumed to be non-atomic and Borel}.

We will now describe a procedure that would allow us to extend a measures $m$, which is initially defined on cylinder sets of $X_B$, to the sigma-algebra $\mathcal B(X_B)$ of all Borel sets.

Let $X$ be a compact metric space and let  $\mathcal B = \mathcal B(X)$ be the sigma-algebra of all Borel sets. Suppose that  $\mathcal F$ and $\mathcal G$ are  two finitely generated  sigma-subalgebras such that $\mathcal F \subset \mathcal G$. Denote by $\mathcal{M(F)}$ and $\mathcal{M(G)}$ the corresponding algebras of $\mathcal F$-measurable and $\mathcal G$-measurable functions on $X$; let $\mathbb I$ denote the constant function ``one'' on $X$.

A positive operator $\mathcal E = \mathcal E_{\mathcal{FG}} : \mathcal {M(G)} \to \mathcal {M(F)}$ is said to be a \textit{conditional expectation} if

(i) $\mathcal E(\mathbb I) = \mathbb I$;

(ii) $\mathcal E$ is positive, that is $\mathcal E$ maps positive functions in $\mathcal{M(G)}$ onto positive functions in $\mathcal{M(F)}$;

(iii) $\mathcal E(fg) = f \mathcal E(g)$ hold for all $g \in \mathcal{M(G)}$ and $f \in \mathcal{M(F)}$.

\begin{lem}\label{measure extension} Let  $(X, \mathcal B)$ be as above, and  let $(\mathcal F_n)_{n\in \mathbb N}$ be a sequence of finitely generated  sigma-subalgebras such that $\mathcal F_n \subset \mathcal F_{n+1}$, and let $\mathcal E_{n} : \mathcal {M}(\mathcal F_{n+1}) \to \mathcal {M}(\mathcal F_n)$ be an associated sequence of conditional expectations such that the following property holds for all $k < l < n$: if $\mathcal E_{j,i} =\mathcal E_i \circ \cdots \circ \mathcal E_{j-1}$, then
$$
\mathcal E_{n,l}\mathcal E_{l,k} = \mathcal E_{n,k}
$$
(in the above notation, $\mathcal E_n = \mathcal E_{n+1, n}$), or, in other words, the following diagram is commutative
$$
\begin{array}[c]{ccc}
 \mathcal {M}(\mathcal F_n)  & \stackrel{\mathcal E_{n,k}}{\longrightarrow} &  \mathcal {M}(\mathcal F_k)\\
\searrow\scriptstyle{\mathcal E_{n,l}}
&&\nearrow\scriptstyle{\mathcal E_{l,k}}\\
&\mathcal {M}(\mathcal F_l)&
\end{array}
$$
Assume further that
 $$
\bigcup_n \mathcal F_n = \mathcal B.
$$
Let $(\mu_n)$ be a sequence of measures, $\mu_n$ defined on $\mathcal F_n$ for all $n \in \mathbb N$, and assume that the conditional expectations satisfy
\begin{equation}\label{expectations1}
     \mathcal E_n(\mu_{n+1}) = \mu_{n}, \ \ \ n \in \mathbb N.
\end{equation}
Then there is a unique measure $\widetilde \mu$ on $\mathcal B$ such that
\begin{equation}\label{expectations2}
      \mathcal E_{\mathcal F_n, \mathcal B}(\widetilde \mu) = \mu_n, \ \ \ n \in \mathbb N.
\end{equation}
\end{lem}

\begin{proof} We give a sketch of the proof. Let $C(X)$ denote the space of continuous functions. Define $\mathcal A = \bigcup_n C(X) \cap\mathcal M(\mathcal F_n)$. Then $\mathcal A$ is closed under the complex conjugacy and separates points in $X$. Hence, it is uniformly dense in $C(X)$ by the Stone-Weierstrass theorem.

For $f \in \mathcal A$, pick $n \in \mathbb N$ such that $f \in \mathcal M(\mathcal F_n)$, and set
\begin{equation}\label{functional L}
     L(f) =\int_X f d\mu_n
\end{equation}
Using (iii) of the definition of the conditional expectation and relation (\ref{expectations1}), we note that $L$ is a well defined linear functional on $\mathcal A$. Indeed, for any $f \in \mathcal M(\mathcal F_n)$,
$$
\int_X f d\mu_{n+1} = \int_X \mathcal E_n(f) d\mu_{n+1}  = \int_X f d(\mathcal E_n(\mu_{n+1})) = \int_X f d\mu_{n}.
$$
Now using the Stone-Weierstrass theorem on $\mathcal A$, we note that $L$ in (\ref{functional L}) extends by closure to $C(X)$. Denote this uniquely defined extension by $\widetilde L$. The Riesz' theorem applied to $\widetilde L$ yields a unique  probability measure $\widetilde \mu$, defined on $\mathcal B$, such that
$$
 \widetilde  L(f) =\int_X f d\widetilde \mu
$$
holds for all $f \in C(X)$.
By standard arguments, we can show that $\mathcal E(\widetilde \mu) = \mu_n,\ \forall n\in \mathbb N$. The result then follows.
\end{proof}

\begin{remark} In order to illustrate the described above method, one can  consider, for instance, the case of the infinite Cartesian product $(X, \nu) = (\prod_i X_i, \prod_i\nu_i)$ of compact measure spaces with $\mathcal F_n = \mathcal B(X_1) \times \cdots \times \mathcal B(X_n)$. Then $\mathcal M(\mathcal F_n)$ is formed by functions $f(x)$ depending on the first $n$ coordinates, i.e. $f(x) = f_n(x_1, ... , x_n)$. The conditional expectation $\mathcal E_n : \mathcal M(\mathcal F_{n+1}) \to \mathcal M(\mathcal F_n)$ is defined by
$$
(\mathcal E_n (f_{n+1}))(x_1, ... , x_{n}) = \int_{X_{n+1}} f_{n+1}(x_1, ... , x_{n+1})d\nu_{n+1}(x_{n+1}).
$$
A direct computation shows that for the measure $\mu_n = \nu_1 \times \cdots \times \nu_n$ on $\mathcal F_n$, we have $\mathcal E_n (\mu_{n+1}) = \mu_n$
and Lemma \ref{measure extension} is applicable.

Another sort of examples that explains the result of the lemma is  based on Bratteli diagrams. Given a Bratteli diagram $B = (V,E)$, define $\mathcal F_n$ as the sigma-subalgebra generated by cylinder sets $[\overline e] = [(e_0, ... , e_n)]$ of length $n$. Then the measure $\mu_n$ on $\mathcal F_n$ can be computed by formulas given, for example,  in (\ref{inv_mesaure}) or (\ref{m([e])}). It is not hard to verify that again one has that the relation  $\mathcal E_n (\mu_{n+1}) = \mu_n$ holds for Markov measures (they are defined below).
\end{remark}

In what follows, we will consider some specific classes of measures on Bratteli diagrams.

\textbf{$\mathcal R$-invariant measures}. Let $B$ be a Bratteli diagram with sequence of incidence matrices $(F_n)$.
It is said that a Borel measure $\mu$ on $X_B$ is $\mathcal{R}$-\textit{invariant} if for any $n\in \mathbb N$, any vertex $w\in V_n$, and any paths $\overline{e}$ and $\overline{e}'$ from $E(v_{0}, w)$  one has $\mu(X_{w}^{(n)}(\overline{e})) = \mu(X_{w}^{(n)}(\overline{e}'))$. Given an invariant measure $\mu$, we set $\mu^{(n)} = (\mu^{(n)}_v : v \in V_n)$ where $\mu^{(n)}_v =\mu(X_{v}^{(n)})$. Then $\mu$ is completely determined by a sequence of positive probability vectors $(\mu^{(n)})$ satisfying the property
\begin{equation}\label{mu^n vectors}
A_n \mu^{(n+1)} = \mu^{(n)},  \ \ \  n\in \mathbb N.
\end{equation}
Denote by $\mathcal M(\mathcal R)$ and $\mathcal M_1(\mathcal R)$ the sets
of  $\mathcal R$-invariant measures and of probability  $\mathcal R$-invariant measures, respectively. If $B$ is a simple Bratteli diagram, then these sets coincide.

In this paper we mostly deal with simple stationary Bratteli diagrams. Let $B$  be a stationary Bratteli diagram  defined by a primitive matrix $A = F^T$. Suppose that $\lambda$ is the Perron-Frobenius eigenvalue of $A$, and  $\mathbf x =  (\mathbf x_1, ... , \mathbf x_K)^T$ is the corresponding strictly positive eigenvector normalized by the condition $\sum_{i=1}^K \mathbf x_i = 1$.
It is well known that, for a simple stationary Bratteli diagram, there exists a unique ergodic $\mathcal R$-invariant measure $\mu$ on $X_B$, that is $\mathcal M_1(\mathcal R) =\{\mu\}$. This measure $\mu$ is completely determined by its values on cylinder sets
\begin{equation}\label{inv_mesaure}
\mu(X_{i}^{(n)}(\overline{e})) = \frac{\mathbf x_i}{\lambda^{n - 1}},
\end{equation}
where $i\in V_n$, and $\overline e$ is a finite path with $r(\overline e) = i$.
\medskip

\textbf{Markov measures}. More generally, we can consider a class of Borel probability measures on the path space $X_B$ called \textit{Markov measures} because of a clear analogue with the case of Markov chains. In particular, this class contains all $\mathcal R$-invariant probability measures  \cite{BJ14}.

Let $(P_n)$ be a sequence of non-negative matrices with entries $(p^{(n)}_{v,e})$ where $v \in V_n, e \in E_{n+1}, n=0, 1, ... $ Thus, the size of $P_n$ is   $|V_n| \times |E_{n+1}|$. In particular, $P_0$ is a row vector. To define a Markov measure $m$, we require that the sequence $(P_n)$ satisfies the following properties:
\begin{equation}\label{defn of P_n}
(a)\ \ p^{(n)}_{v,e} > 0 \ \Longleftrightarrow \ s(e) = v; \ \ \ \ (b)\ \  \sum_{e : s(e) = v} p^{(n)}_{v,e} =1.
\end{equation}
 Then we set for any cylinder set $[\overline e] = [(e_0, e_1, ... , e_n)]$
\begin{equation}\label{m([e])}
m([\overline e]) = p^{(0)}_{v_0, e_0}p^{(1)}_{s(e_1), e_1} \cdots p^{(n)}_{s(e_n), e_n}.
\end{equation}
To emphasize that $m$ is generated by a sequence of stochastic matrices, we will also write down $m = m(P_n)$.

Now we return to the definition of a Markov measure and show that, for any such a measure  $m(P_n)$ on a Bratteli diagram $B = (V,E)$,  we can inductively define a sequence of probability vectors $q^{(k)} = (q^{(k)}(v) : v \in V_k),\ k \geq 1$, by the following formula
$$
q^{(k)}(v)  = \sum_{e \in E_k : r(e) = v} q^{(k-1)}(s(e)) p^{(k)}_{s(e), e}
$$
where $P_k = (p^{(k)}_{s(e),e})$ and $q^{(0)} = 1$.

Let $\nu$ be a probability $\mathcal R$-invariant measure. Then one can show that in this case the vectors $q^{(k)}$ coincide with $\nu(X_v^{(k)})$ defined in (\ref{mu^n vectors}).

The following result makes a link between the two  classes of measures.

\begin{lem} [\cite{BJ14}]\label{inv meas is Markov}
Let $\nu \in \mathcal M_1(\mathcal R)$. Then there exists a sequence of stochastic matrices $(P_n)$ such that $\nu =  m(P_n)$. In other words, every probability Borel $\mathcal R$-invariant measure is a Markov measure.
\end{lem}

For a stationary Bratteli diagram $B$, it is natural to distinguish and study a special subset of Markov measures $\nu = \nu(P)$, the so called \textit{stationary Markov measures}. They are obtained when all matrices $P_n, n\in \mathbb N,$ are the same and equal to a fixed matrix  $P$. Formula (\ref{m([e])}) is transformed then as follows:
\begin{equation}\label{nu(P)}
\nu([\overline e]) = p^{(0)}_{v_0, e_0}p_{s(e_1), e_1} \cdots p_{s(e_n), e_n}.
\end{equation}

\subsection{Graphs coupled with Bratteli diagrams} In this subsection we show how can one associate  a directed graph $G = (T,P)$ to a stationary Bratteli diagram. It will be clear that the suggested construction can be used in more general settings but we are focused here on the case of stationary diagrams only. Moreover, without loss of generality, we assume that a given stationary Bratteli diagram $B$ is a 0-1 simple diagram; that is it satisfies conditions (1) and (2) of Lemma \ref{0-1_matrix}.

Let $E$ be the edge set  between the first and second levels of $B$ and let $A$ be the transpose of the incidence matrix. By the made assumption, the diagram has only single edges between the vertices of consecutive levels.

\begin{remark}\label{correspondence E=A} We note that there is a one-to-one correspondence between non-zero entries of $A$ and edges of $E$:
$$
a_{i,j}\  \longleftrightarrow \ e \ \ \mbox{iff} \ \ s(e) = i, r(e) = j,\  i,j \in V.
$$
This simple observation will be regularly exploited below.
\end{remark}

\begin{definition}\label{linked edges} (1) We say that a \textit{pair of edges $(e, f) \in E\times E$ is linked if $r(e) = s(f)$}. Denote by $\mathcal L(E)$ the set of linked pairs.

(ii) Let $e \longleftrightarrow a_{i,j}$ and $f \longleftrightarrow a_{k,l}$ be the correspondences defined  in Remark \ref{correspondence E=A}. Then $(e,f) \in \mathcal L(E)$ if and only if $k = j$. In this case, we say that \textit{$a_{j,l}$ follows $a_{i,j}$.}
\end{definition}

Next, we want to associate a directed graph to a stationary Bratteli diagram with  0-1 matrix $A$. This graph, $G$, will be uniquely defined by the matrix $A$ so that we can write down $G= G(A)$.

\begin{definition}\label{graph G(A)} Let $A$ be a 0-1 matrix. Then the set of vertices, $T$, of  the directed graph $G = G(A)$ is formed non-zero entries $a_{i,j}$ of $A$. To define the set of directed edges, $P$, of $G$, we say that there is an arrow (directed edge) from $a_{i,j}$ to $a_{k,l}$ if and only if  $a_{k,l}$ follows $a_{i,j}$ (i.e. $j=k$). By definition,  $G =(T,P)$ is called a {\em coupled graph}.
 \end{definition}

It follows from this definition that for a fixed non-zero entry $a_{i,j} $ of $A$ (or a vertex $t \in T$) the number of incoming edges for $t$ equals the number of non-zero entries of $A$ in the $i$-th column, and the number of outgoing edges for $t$ equals the number of non-zero entries of $A$ in the $j$-th row.

Any graph which is isomorphic to $G(A)$ can be treated as
a  graph coupled to $A$.

\begin{example}\label{example 3x3} Let
$$
A =\left(
     \begin{array}{ccc}
       1 & 1 & 0 \\
       0 & 1 & 1\\
       1 & 0 & 1 \\
     \end{array}
   \right).
$$
Then $G(A)$ can be represented as follows.
$$
\unitlength=1cm
\begin{graph}(8,6.5)
 \graphnodesize{0.2}
 \roundnode{V11}(1,5)
  \roundnode{V12}(4,5)
 \roundnode{V22}(4,3)
 \roundnode{V23}(7,3)
 \roundnode{V31}(1,1)
 \roundnode{V33}(7,1)
 %
 %
   \diredge{V11}{V12}
\dirloopedge{V11}{90}(-0.5,0.5)
\diredge{V12}{V22}
 \diredge{V12}{V23}
\dirloopedge{V22}{-90}(-0.5,-0.5)
 \diredge{V22}{V23}
 \diredge{V23}{V31}
 \diredge{V23}{V33}
 \diredge{V33}{V31}
\dirloopedge{V33}{90}(0.5,-0.5)
 \diredge{V31}{V12}
  \diredge{V31}{V11}
 \end{graph}
$$
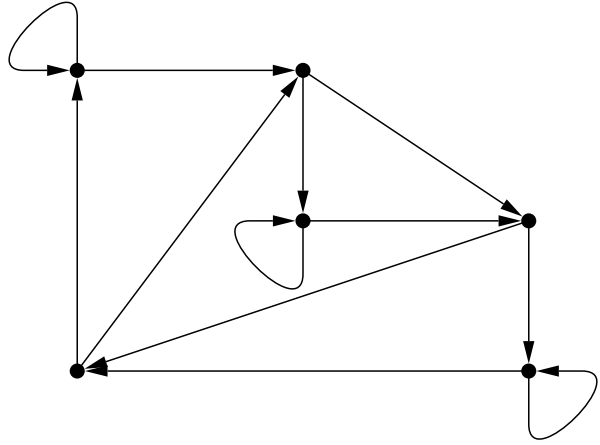
\hskip 3cm \mbox{Fig. 1: Graph defined by the matrix $A$}
\end{example}
 Obviously, such a representation of $G(A)$ is  unique up to an isomorphism.

We recall that  a graph $G$ is called strongly connected if for any two vertices $t_1$ and $t_2$ of $G$ there exist a path from $t_1$ to $t_2$ and a path from $t_2$ to $t_1$.

\begin{prop}\label{graph properties} Suppose that $B$ is  a stationary 0-1 Bratteli diagram. Then

(1) the diagram is simple if and only if the coupled graph $G$ is strongly connected;

(2) there is a one-to-one correspondence between the path space $X_B$ of $B$ and the set of infinite paths $X_G$ of the coupled graph $G$.
\end{prop}

\begin{proof} (1) Suppose that $B$ is simple. In other words, this means that $A$ is primitive. We need to show that for any two vertices $t, t'  \in T$ there are  paths from $t$ to $t'$ and from $t'$ to $t$. Let $t$ correspond to $a_{i,j}$ and let $t'$ correspond to $a_{i',j'}$. By simplicity of $B$, there exists a finite path $\overline e = (e_1, ... , e_{n})$ in the path space $X_B$ such that $s(e_1) = i, r(e_1) =j$, and $s(e_n) = i', r(e_{n}) = j'$. According to Remark \ref{correspondence E=A}, the path $\overline e$ determines uniquely a sequence of non-zero entries $a_{s(e_1), r(e_1)}, a_{s(e_2), r(e_2)}, ... , a_{s(e_{n}), r(e_{n})}$  of $A$. By Definition \ref{graph G(A)}, this sequence corresponds to a path in $G(A)$ that starts at $t$ and ends at $t'$. Similarly, one can show that there exists a path from $t'$ to $t$.

(2) This statement follows, in fact, from (1) because the same method  can be applied to any infinite path.
\end{proof}

\begin{remark}\label{m-x assoc to G(A)}
Let $\Gamma$ be a directed graph with the set of vertices denoted by $W$. Then one can associate a 0-1 matrix $\overline A = (\overline a_{v,w})$ of the size $|W|\times |W|$, which is usually called an adjacency matrix. By definition, $\overline a_{v,w} =1$ if and only if there exists a directed edge $e$ in $\Gamma$ from the vertex $v$ to the vertex $w$. In the case when $\Gamma = G(A)$ for some 0-1 matrix $A$ (a stationary Bratteli diagram $B$ in other words), we observe that, because of the identification of non-zero entries of $A$ and edges from $E$, the adjacency matrix has the size $|E| \times |E|$ is determined by the rule $(a_{e,f} =1) \Longleftrightarrow (r(e) = s(f))$.
\end{remark}

{\em Question}. It would be interesting to find out what kind of directed strongly connected graphs are isomorphic to graphs obtained from 0-1 stationary simple Bratteli diagrams.

\subsection{Semibranching function systems}  Here we give the definition of a semibranching function system following \cite{MaPa11}.

\begin{definition}\label{s.f.s.} (1) Let $(X,\mu)$ be a probability measure space with non-atomic measure $\mu$. We consider a finite family $\{\sigma_i : i\in \Lambda\}$ of one-to-one $\mu$-measurable maps $\sigma_i$ defined on a subset $D_i$ of $X$ and let $R_i = \sigma_i(D_i)$. The family  $\{\sigma_i\}$ is called a \textit{semibranching function system  (s.f.s.)} if the following conditions hold:

(i) $\mu(R_i \cap R_j) = 0$ for $i\neq j$ and $\mu(X\setminus \bigcup_{i\in \Lambda} R_i) = 0$;

(ii) $\mu \circ \sigma_i << \mu$ and
$$
\rho{_\mu}(x, \sigma_i) := \frac{d\mu\circ \sigma_i}{d\mu}(x) > 0 \ \
 \mbox{for $\mu$-a.e. $x\in D_i$};
$$

(iii) there exists an endomorphism $\sigma : X \to X$ (called a \textit{coding map}) such that $\sigma\circ \sigma_i(x) = x$ for $\mu$-a.e. $x\in D_i, \ i \in \Lambda$.

If, additionally to properties (i) - (iii), we have $\bigcup_{i\in \Lambda}D_i = X$ ($\mu$-a.e.),  then  the s.f.s. $\{\sigma_i : i\in \Lambda\}$ is called \textit{saturated.}

(2) We also say that a saturated s.f.s. satisfies \textit{condition (C-K)} if for any $i\in \Lambda$ there exists a subset $\Lambda_i \subset \Lambda$ such that up to a set of measure zero
$$
D_i = \bigcup_{j\in \Lambda_i} R_j.
$$
In this case, condition (C-K) defines a 0-1 matrix $\wt A$ by the rule:
\begin{equation}\label{C-K defines A}
\wt a_{i,j} =1 \ \ \Longleftrightarrow \ \ j\in \Lambda_i, \ \ i \in \Lambda.
\end{equation}
Then the matrix $\wt A$ is of the size $|\Lambda|\times |\Lambda|$.
\end{definition}
\medskip

The following two examples of  s.f.s. (see Examples \ref{example of s.f.s.} and \ref{ex s.f.s. topol Markov chain}), which are generated by a stationary Bratteli diagram, will play the key role in our constructions.

\begin{example}\label{example of s.f.s.} Let $B$ be a stationary simple 0-1 Bratteli diagram. We construct a s.f.s. $\wt \Sigma$ defined on the path space $X_B$ endowed with a Markov measure $m$ which, as we will see below, must have some additional properties to satisfy Definition \ref{s.f.s.}. This s.f.s. is determined by the edge set $E$ which is the set of edges between any two consecutive levels of $B$. This set plays  the role of the index set  $\Lambda$ that was used in Definition \ref{s.f.s.}. For any $e \in E$, we denote
\begin{equation}\label{defn of D_e}
D_e = \{y = (y_i) \in X_B : s(y_1) = r(e)\},
\end{equation}
\begin{equation}\label{defn of R_e}
 R_e = \{y = (y_i) \in X_B : y_1 = e\}.
\end{equation}
Then we see that $D_e$ depends on $r(e)$ only, so that $D_e = D_{e'}$ if $r(e) = r(e')$, and $D_e \cap D_{e'} = \emptyset$ if  $r(e) \neq r(e')$.

To define a s.f.s.  $\{\sigma_e : e \in E\}$ we consider the map
\begin{equation}\label{defn of sigma_e}
     \sigma_e (y) :=  (y_0', e, y_1, y_2, .... )
\end{equation}
is a one-to-one continuous map from $D_e$ onto $R_e$. Here the edge $y_0'$ is uniquely determined by $e$ as the edge connecting $v_0$ and $s(e)$.   Let  $\sigma : X_B \to X_B$ be also defined as follows: for any $x = (x_i)_{i \geq 0} \in X_B$
\begin{equation}\label{defn of sigma}
     \sigma(x) :=  (z_0', x_2, x_3, ...)
\end{equation}
where again $z'_0$ is uniquely determined by the vertex $s(x_2)$.
Then it follows from (\ref{defn of sigma_e}) and (\ref{defn of sigma}) that the map $\sigma$ is onto and satisfies the relation
$$
\sigma\circ\sigma_e(x) = x,\ \ \  x \in D_e;
$$
 hence $\sigma$ is a coding map.

We immediately deduce from (\ref{defn of R_e}) that $\{R_e : e\in E\}$ constitutes a partition of $X_B$ into clopen sets. Relation (\ref{defn of D_e}) implies that the s.f.s. $\{\sigma_e : e \in E\}$ is a saturated s.f.s. Moreover, we claim that it satisfies condition (C-K), that is
\begin{equation}\label{checking C-K}
     D_e = \bigcup_{f : s(f) = r(e)} R_f,\ \ \ e \in E.
\end{equation}
Indeed, $y = (y_i) \in D_e \ \Longleftrightarrow \ s(y_1) = r(e) \ \Longleftrightarrow \ \exists f = y_1\ \mbox{such\ that}\ y = (y_0, f, y_2, ...) \ \Longleftrightarrow \ y \in \bigcup_{f : s(f) = r(e)} R_f$.

Relation (\ref{checking C-K}) shows that the non-zero entries of the  0-1 matrix $\wt A$ from Definition \ref{s.f.s.} are defined by the rule:
$$
(\wt a_{e,f} =1) \ \Longleftrightarrow \ (s(f) = r(e)) \ \Longleftrightarrow \ ((e, f) \in \mathcal L(E)).
$$

In order to clarify the nature of the matrix $\wt A$, we observe that $\wt A$ coincides with the adjacency matrix of the graph $G(A)$ constructed by the initial matrix $A$ (see Remark \ref{m-x assoc to G(A)}).

Next, we observe that $\sigma: X_B \to X_B$ is a finite-to-one continuous map. Indeed, given $x = (x_i) \in X_B$, one can verify that
$$
|\sigma^{-1}(x)| = |r^{-1}(r(x_1))| = \sum_{u\in V} f_{v,u}.
$$
In other words, we see that $\sigma(x) = \sigma(y)$ if and only if $r(x_1) = r(y_1)$ where $x = (x_i), y = (y_i)$, and   if $x \in D_{e_1} = \cdots =  D_{e_k}$, then $\sigma^{-1}(x) = \{z_1, ... , z_k\}$ where $z_i \in R_{e_i}$.

Thus, it remains to find out under what conditions property (ii) of Definition \ref{s.f.s.} holds. In other words, the Radon-Nikodym derivative $\rho_m(x, \sigma_e)$ must be  positive on the set $D_e$ for $\mu$-a.e. $x\in D_e$. Since the path space $X_B$ is naturally partitioned into a refining sequence of clopen partitions formed by cylinder sets of fixed length, we can apply de Possel's theorem (see, for instance, \cite{SG77}). Let $\mathcal Q_n$ be a partition of $X_B$ into cylinder sets $[\overline e]$ where each finite path $\overline e$ has length $n$. Suppose that $m$ is a Borel probability measure on $X_B$ and $\varphi$ is a measurable one-to-one map on $X_B$. By de Possel's theorem, we have for $\mu$-a.a $x$,
$$
\rho(x, \varphi) = \lim_{n\to \infty} \frac{m(\varphi([\overline e(n)])}{[\overline e(n)]}
$$
where $\{x\} = \bigcap_n [\overline e(n)]$.

We first consider  the case when $m$ is the unique $\mathcal R$-invariant measure $\mu$.
If $\overline f = (f_0, f_1, ... , f_n) \in D_e$, then $\sigma_e(\overline f) = (f'_0, e, f_1, ... , f_n)$. By (\ref{inv_mesaure}), we obtain for $\mu$-a.e. $x\in D_e$
$$
\mu([\overline f ]) = \frac{\mathbf x_{r(\overline f)}}{\lambda^{n-1}},\ \ \ \mu(\sigma_e([\overline f ])) = \frac{\mathbf x_{r(\overline f)}}{\lambda^{n}}
$$
and therefore
\begin{equation}\label{RN for mu}
     \rho_{\mu}(x, \sigma_e) = \lambda^{-1}.
\end{equation}

In case of a stationary Markov measure $\nu = \nu(P)$, we see from (\ref{nu(P)}) that if
$\{x\} =\bigcap_n [f(n)] \in D_e$ then
$$
\nu([\overline f(n)]) = p^{(0)}_{v_0, f_0}p_{s(f_1), f_1} \cdots p_{s(f_n), f_n},
$$
$$
\nu(\sigma_e([\overline f(n)])) = p^{(0)}_{v_0, f'_0}p_{s(e), e}p_{s(f_1), f_1} \cdots p_{s(f_n), f_n},
$$
and finally
\begin{equation}\label{RN for nu}
 \rho_{\nu}(x, \sigma_e) = \frac{p^{(0)}_{v_0, f'_0}p_{s(e),e}}{p^{(0)}_{v_0, f_0}}
\end{equation}
(the meaning of $f'_0$ was explained above).

It follows from (\ref{RN for mu}) and (\ref{RN for nu}) that the Radon-Nikodym derivatives  $\rho_{\mu}$ and $\rho_{\nu}$ are positive on $D_e$ if and only if all entries of the vector $p_0$ are positive. The latter, in particular, means that the support of $\nu$ is the whole space $X_B$.

In case of an arbitrary Markov measure $m$, we obtain more restrictive conditions under which the Radon-Nikodym derivative  $\rho_m(x, \sigma_e)$ is positive on $D_e$. If $x = (x_i)\in D_e$ is determined by the sequence $[\overline f(n)] = [(f_0, f_1, ... , f_n)]$ such that $x_i = f_i, i= 0,1, ... ,n$, we see that
$$
m([f(n)]) = p^{(0)}_{v_0, f_0}p^{(1)}_{s(f_1), f_1} \cdots p^{(n)}_{s(f_n), f_n}
$$
and
$$
m(\sigma_e([f(n)])) = p^{(0)}_{v_0, f'_0}p^{(1)}_{s(e), e}p^{(2)}_{s(f_1), f_1} \cdots p^{(n+1)}_{s(f_n), f_n}.
$$
\begin{lem}\label{RN for m}
Let $m$ be a Markov measure on the path space of a stationary 0-1 Bratteli diagram, then $\rho_m(x, \sigma_e) >0$ on $D_e$ if and only if
\begin{equation}\label{product convergence}
   0<  \prod_{i=1}^\infty \frac{p^{(i+1)}_{s(f_i), f_i}}{p^{(i)}_{s(f_i), f_i}}  <\infty
\end{equation}
for any $x = \bigcap_n[\overline f(n)] \in D_e$
\end{lem}

\begin{proof} This results follows immediately from the relation
\begin{eqnarray*}
  \frac{dm\circ\sigma_e}{dm}(x) &=& \lim_{n\to \infty}\frac{m(\sigma_e([f(n)]))}{m([f(n)])} \\
   &=& \lim_{n\to\infty} \frac{p^{(0)}_{v_0, f'_0}p^{(1)}_{s(e), e}p^{(2)}_{s(f_1), f_1} \cdots p^{(n+1)}_{s(f_n), f_n} }{p^{(0)}_{v_0, f_0}p^{(1)}_{s(f_1), f_1} \cdots p^{(n)}_{s(f_n), f_n}}\\
   &=& \frac{ p^{(0)}_{v_0, f'_0}p^{(1)}_{s(e), e}}{p^{(0)}_{v_0, f_0} } \prod_{i=1}^\infty \frac{p^{(i+1)}_{s(f_i), f_i}}{p^{(i)}_{s(f_i), f_i}}.
\end{eqnarray*}
\end{proof}

A Markov measure satisfying (\ref{product convergence}) is called a quasi-stationary measure. We remark that condition (\ref{product convergence}) appeared first in \cite{DJ14b} in a different context.
\end{example}

Based on the above discussion, we can summarize the mentioned results  in the following theorem.

\begin{thm}\label{example summary}
Given a 0-1 stationary simple Bratteli diagram $B$ with edge set $E$, the collection of maps $\{\sigma_e : D_e \to R_e\},\ e \in E$,  defined as in Example \ref{example of s.f.s.} on the space  $(X_B, m)$, forms  a saturated s.f.s. $\wt \Sigma$ satisfying (C-K) condition where the Markov measure $m$ is either the unique $\mathcal R$-invariant measure $\mu$, or a stationary Markov measure $\nu$ of full support, or  a quasi-stationary measure Markov measure.
\end{thm}

\begin{example}\label{ex s.f.s. topol Markov chain} We now give the other example of a s.f.s., denoted by $\Sigma$, which is naturally related to a stationary 0-1 Bratteli diagram $B$ with the incidence matrix $A$. This sort of examples was first  considered in \cite{MaPa11}.

Let $A = (a_{i,j}  :  i, j =1, ... , n)$ be a 0-1 primitive matrix, and let $X_A$ be the corresponding topological Markov chain where  $(x \in X_A) \ \Longleftrightarrow\ (x = (x_j)_{j\geq 1} : a_{x_j, x_{j+1}} =1)$. For each $i = 1, ... , n$,  we define
$$
D_i = \{ x = (x_k) \in X_A : a_{i, x_1} =1 \}, \ \ R_i = \{ x = (x_k) \in X_A : x_1 = i\},
$$
$$
\sigma_i : D_i \to R_i, \ \ \ \sigma_i(x) = ix,
$$
and
$$
\sigma : X_A \to X_A, \ \ \ \sigma(x_1,x_2, ... , x_k, ...)  = (x_2, x_3, ... , x_k, ...).
$$
Then $D_i = \bigcup_{j : a_{i,j} =1} R_j$ for each $i$. This means that the matrix associated to the s.f.s. $\Sigma$ according to (\ref{C-K defines A}) is exactly $A$.

In order to finish the definition of the s.f.s. $\Sigma$, one needs to specify a measure $m$ on $X_A$ satisfying Definition \ref{s.f.s.}. To do this, one can take, for instance, the Hausdorff measure (as was done in \cite{MaPa11}) or some Markov measures analogous to those  considered in Example \ref{example of s.f.s.}. We will not discuss the details here.
\end{example}

\begin{remark}\label{discussion of the two exs of sfs}
In Examples \ref{example of s.f.s.} and \ref{ex s.f.s. topol Markov chain} we defined two s.f.s., $\wt \Sigma$ and $\Sigma$, built by a  0-1 stationary Bratteli diagram $B$ with the incidence matrix $A$. Here we discuss some relations between these two s.f.s. To use more consistent notation, we will add the symbol `tilde', to any object of $\wt\Sigma$.

We first observe that these s.f.s. can be thought to be defined on the same space $X_B$ because the sets $X_A$ and $X_B$ are naturally homeomorphic. We have $X_A = (x_1, x_2,  \cdots ) \mapsto (e_0, e_1, e_2, \cdots ) \in X_B$ where the edge $e_i$ is uniquely defined by the properties $s(e_i) = x_i, r(e_i) = x_{i+1}, i \geq 0$, with $x_0 = v_0$. Thus, we identify these spaces of infinite paths and consider  $\wt \Sigma$ and $\Sigma$ as the s.f.s. defined on the same space $X_B$.

We recall that the index sets of $\wt \Sigma$ and $\Sigma$ are the edge set $E$ and the vertex set $V$ of $B$, respectively. It follows from the definitions of $\wt \Sigma$ and $\Sigma$, that for every $i\in V$
$$
R_i = \bigcup_{e \in s^{-1}(i)} \wt R_e, \ \ \ D_i = \bigcup_{e : a_{i,s(e)} =1} \wt D_{e}.
$$
Next, we see that the transformations $\wt \sigma$, $\sigma$, $\wt \sigma_e$, and $\sigma_i$ are related as follows:
$$
 \sigma_i|_{\wt D_e} = \wt\sigma_e,\ \ s(e) = i \in V,
$$
and $\sigma = \wt \sigma$. Moreover, it follows from the above relations that if $\mu$ is a measure on $X_B$ satisfying (ii) of Definition \ref{s.f.s.} for $\wt\Sigma$, then the measure $\mu$ satisfies  the same property for the s.f.s. $\Sigma$.

The relations between the objects of  $\wt\Sigma$ and $\Sigma$ show that $\wt\Sigma$ refines the subsets that constitute $\Sigma$, and moreover the restriction of the $\sigma_i$ on the corresponding subsets $\wt D_e$ coincides with $\wt\sigma_e$. In this case we will say that $\wt\Sigma$ is a {\em refinement} of $\Sigma$.

In order to illustrate this notion, we consider the example of Bratteli diagram with incidence matrix
$$ A =\left(
                    \begin{array}{ccc}
                      1 & 1 & 0 \\
                      0 & 1 & 1 \\
                      1 & 0 & 1 \\
                    \end{array}
                  \right).
$$
Let $1, 2, 3$ stand for vertices of $B$. Then we have the following relations
$D_1 = \wt D_1 \cup \wt D_2$, $D_2 = \wt D_2 \cup \wt D_3$, and $D_3 = \wt D_1 \cup \wt D_3$ (we remind that $\wt D_e = \wt D_{r(e)}$). Similarly, $R_i =
\bigcup_{e : s(e) =i}  \wt R_e$, $i =1, 2, 3$. Finally, if one computes, for instance,  $\sigma_1|_{\wt D_2}$, then it is the same as $\sigma_e$ where $r(e) =2$.

It is clear that the described above situation can happen in more abstract settings when two s.f.s., say $\wt \Gamma = (\wt \gamma_{\wt\omega} : \wt D_{\wt\omega} \to \wt R_{\wt\omega},\ \wt\omega \in \wt\Omega)$ and $\Gamma = ( \gamma_{\omega} :  D_{\omega} \to R_{\omega}, \ \omega \in \Omega  )$, are defined on a measure space $(X, m)$. Then we say that $\wt \Gamma$ {\em refines } $\Gamma$ if the partition of $X$ formed by $(\wt R_{\wt\omega})$ refines that formed by $(R_{\omega})$, and  for every $\omega\in \Omega$
$$
D_\omega = \bigcup_{\wt\omega \in \Lambda_\omega} \wt D_{\wt\omega}, \ \
R_\omega = \bigcup_{\wt\omega \in \Xi_\omega} \wt R_{\wt\omega},
$$
$$
\gamma_\omega |_{\wt D_{\wt\omega}} = \wt\gamma_{\wt\omega}, \ \ \wt\omega \in \Lambda_\omega.
$$

Suppose that the s.f.s. $\Gamma$ and $\wt \Gamma$ both satisfy (C-K) condition,  and therefore they define two matrices $A$ and $\wt A$ respectively.  We will see below (see Theorem \ref{CK-alg repr by sfs}) that any s.f.s. with (C-K) condition  $\Gamma$  generates a representation of $\mathcal  O_A$ on $L^2(X,m)$. In the described case, we get representations of $\mathcal O_A$ and $\mathcal O_{\wt A}$.  It would be interesting to find out how these representations are related each other  when one of these s.f.s. refines the other one. We answer this question for s.f.s. $\wt\Sigma$ and $\Sigma$ in Section \ref{section isomorphic s.f.s.}.
\end{remark}

\subsection{The Cuntz-Krieger algebra $\mathcal O_A$} \label{C-K algebra}

Let $A$ be a primitive $n\times n$ matrix with 0-1 entries $a_{i,j}$. The \textit{Cuntz-Krieger algebra $\mathcal O_A$} is generated by partial isometries $S_1, ... , S_n$ satisfying the relations
\begin{equation}\label{CK relations}
    \sum_{i=1}^n S_iS^*_i =1,\ \ \ \ S^*_iS_i = \sum_{j=1}^n a_{i,j}S_jS^*_j.
\end{equation}

The Cuntz algebra $\mathcal O_N$ corresponds to the special case of $A$ when all entries are ones, and $S_iS_i^* =1$ for all $i$, i.e., the generators are isometries, as opposed to {\em partial} isometries.

Let $I = i_1\cdots i_k$ be a finite word over  the alphabet $\{1, ..., n\}$. Define  $S_I = S_{i_1}\cdots S_{i_k}$. Let $\mathcal D_A$ be a subalgebra of $\mathcal O_A$ generated by $\{S_IS^*_I : I\ \mbox{is\ any\ finite\ word}\}$. It is well known (see, for instance, \cite{CuKr80}) that $\mathcal D_A$ is isomorphic to the commutative $C^*$-algebra $C(X_A)$ of the complex-valued functions defined on the space $X_A$ of infinite path of the topological Markov chain, i.e.,  $ X_A = ((x_j)_{j\geq 0}  : a_{x_j, x_{j+1}} =1).$

It turns out that any s.f.s. satisfying (C-K) condition is a source for construction of representations of the Cuntz=Krieger algebra.
The next theorem shows how such representations of $\mathcal O_A$ are arisen.

\begin{thm}[\cite{MaPa11}]\label{CK-alg repr by sfs} Let $\{\sigma_i : i \in \Lambda\}$ be a s.f.s. with coding map $\sigma$ defined on a probability measure space $(X,m)$. Suppose that it satisfies  condition (C-K). Let $\wt A$ be a 0-1 matrix defined by relation (\ref{C-K defines A}).  Then the operators $T_i = T_i(m)$ and $T^*_i = T^*_i(m)$ acting on $L^2(X,m)$ by formulas
\begin{equation}\label{T_i formula}
     (T_i\xi)(x) = \chi_{R_i}(x) \rho_m(\sigma(x), \sigma_i)^{-1/2} \xi(\sigma (x)),  \ i\in \Lambda,  \xi\in L^2(X,m)
\end{equation}
and
\begin{equation}\label{T^*_i formula}
  (T^*_i\xi)(x) = \chi_{D_i}(x) \rho_m(x, \sigma_i)^{1/2} \xi(\sigma_i (x)),  \ i\in \Lambda, \  \xi\in L^2(X,m)
\end{equation}
satisfy (\ref{CK relations}) and generate a representation $\pi =\pi(m)$ of $\mathcal O_{\wt A}$.
\end{thm}

It is worth noting that $T_i(m)T^*_i(m)$ is a projection on $L^2(X,m)$ given by the  multiplication operator by $\chi_{R_i}$, and similarly $T^*_i(m)T_i(m)$ is a projection realized by multiplication by $\chi_{D_i}$. When the measure $m$ is clearly understood we will write simply $T_i$ and $T^*_i$, $i \in \Lambda$.

Theorem \ref{CK-alg repr by sfs} will be applied below in the case when the s.f.s.
is taken from Example \ref{example of s.f.s.}.

\begin{prop}\label{eqv meas - eqv repr} Suppose that a s.f.s. $\{\sigma_i : i \in \Lambda\}$ and $(X,m)$ are  as in Theorem \ref{CK-alg repr by sfs}. Let $m'$ be another probability measure equivalent to $m$. Then the representations $\pi(m)$ and $\pi(m')$ defined as in Theorem \ref{CK-alg repr by sfs} are equivalent.
\end{prop}

The \textit{proof} of the proposition is straightforward and contained in a more general result (see Theorem \ref{isomorphic reprs}) proved in the next section so that we can omit it.

Consider  a simple 0-1 stationary Bratteli diagram $B =B(V,E)$. We identify here the set of vertices $V$ with $\{1, ... , n\}$. Then we have two s.f.s. $\Sigma$ and $\wt\Sigma$ described in Examples \ref{ex s.f.s. topol Markov chain} and \ref{example of s.f.s.}.
Let $A$ and $\wt A$ be the corresponding 0-1 matrices defined as in those examples. Thus, we have two Cuntz-Krieger algebras, $\mathcal O_A$ and $\mathcal O_{\wt A}$.

\begin{lem}\label{isom comm algs}
The commutative subalgebras $\mathcal D_A$ and $\mathcal D_{\wt A}$ of $\mathcal O_A$ and $\mathcal O_{\wt A}$ respectively are isomorphic.
\end{lem}

 \begin{proof} We observe that the commutative $C^*$-algebras $C(X_A)$ and $C(X_B)$ are naturally isomorphic. Indeed, we can identify the characteristic functions of cylinder sets from $C(X_A)$ and $C(X_B)$ by the following rule. Let $i_1, ... , i_{n+1}$ be a finite sequence of vertices of $B$ such that $a_{i_k, i_{k+1}} = 1,\ k =1, ... , n$. Then we associate to every pair $(i_k, i_{k+1})$ the uniquely determined edge $e_k$ in $E$ such that $s(e_k) = i_k, r(e_k) = i_{k+1}$. This defines a finite path in $X_B$; we can think that this path begins at $v_0$ since the edge $e_0$ is completely determined by $i_1 = s(e_1)$. It follows that  a one-to-one correspondence
$$
\chi_{[i_1, ... , i_n]} \ \longleftrightarrow \ \chi_{[e_0, e_1, ... , e_n]}
$$
between the characteristic functions of the cylinder sets is well defined and can be extended by linearity on the algebra generated by characteristic functions. Moreover, this algebra is dense in the space of continuous functions.

Next, we  recall that the entries of 0-1 matrix $\wt A$ is enumerated by edges from $E$, so that we can consider the topological Markov chain $X_{\wt A} = \{(e_i) : \wt a_{e_i, e_{i+1}} = 1\}$. It is obvious that there is one-to-one correspondence between elements of $X_{\wt A}$ and $X_B$ because $(x_{e_i}) \in X_B$ if and only if $s(e_{i+1}) = r(e_i)$ if and only if $\wt a_{e_i, e_{i+1}} = 1$. The remaining argument is clear.

\end{proof}


\section{Isomorphic semibranching function systems}

From now on, we will use the following \textit{convention}: all relations on a measure space between functions, transformations, sets, etc. are understood as $\mod 0$ relations; that is they are true up to a set of measure zero. For instance, a measurable map $F : (X, m) \to (X', m')$ is said to be ``onto'' if $m'(X' \setminus F(X)) =0$. Because these properties are obvious as a rule, we  usually omit the words ``almost everywhere'' and $\mod 0$ notation.

\begin{definition}\label{isomorphic sfs}
Let $\{\sigma_i : i \in \Lambda\}$ and $\{\sigma'_i : i \in \Lambda\}$ be two saturated s.f.s. defined on measure spaces
$(X,m)$ and $(X',m')$ respectively. Let  $\sigma : X\to X$ and $\sigma' : X' \to X'$ be the corresponding coding maps. Suppose that there is a family of measurable maps $\{\varphi_i : i\in \Lambda\}$ defined on $R_i \subset X$  such that for all $i \in \Lambda$:

(a) $\varphi_i : R_i \to R'_i $ is one-to-one and onto

(b) $(m'\circ \varphi_i)|_{R_i} \sim  \mu |_{R_i}$, that is $\dfrac{dm'\circ \varphi_i}{dm}(x) > 0$ if and only if $x\in R_i$;

(c) $\varphi_i \circ \sigma_i (x) = \sigma_i' \circ \Phi_i (x)$, $x\in D_i$.\\
Here the map $\Phi_i : D_i \to D'_i$ is defined by the following rule:
$$
\Phi_i(x) = \varphi_j(x), \ \ \forall x\in R_j, \ j\in \Lambda_i,
$$
where $D_i = \bigcup_{j\in \Lambda_i} R_j$. Then we say that $\{\sigma_i : i \in \Lambda\}$ and $\{\sigma'_i : i \in \Lambda\}$ are isomorphic s.f.s.
\end{definition}

\begin{remark}\label{map F}
(1) It follows immediately that, in conditions of Definition \ref{isomorphic sfs}, the collection of maps $\{\varphi_i : i \in \Lambda\}$ determines uniquely a one-to-one onto map $F : X \to X'$ such that
$$
F(x) = \varphi_i(x), \ \ \ x\in R_i, \ \ i \in \Lambda.
$$
Moreover, $\mu'\circ F \sim \mu$ and for $x\in R_i$
$$
\frac{dm'\circ F }{ dm}(x) = \frac{dm'\circ \varphi_i }{dm }(x).
$$
Since $\{R_i : i\in \Lambda\}$ forms a partition of $X$, the Radon-Nikodym derivative $\dfrac{dm'\circ F }{ dm}(x)$ is positive on $X$.

(2) We notice that conditions (a) - (c) of Definition \ref{isomorphic sfs} can be rewritten in terms of $F$; for example, (c) looks as follows:  $F\circ \sigma_i (x) = \sigma_i' \circ F (x)$, $x\in D_i$.

(3) (C-K) condition is invariant with respect to an isomorphism of s.f.s. $\{\sigma_i : i \in \Lambda\}$ and $\{\sigma'_i : i \in \Lambda\}$.
\end{remark}

We will use the following lemma below.

\begin{lem}\label{properties of sfs} Given a s.f.s. as in Definition \ref{s.f.s.}, the following statements hold: for any $i \in \Lambda$
\begin{eqnarray*}
  (1)  & \sigma(R_i) = D_i,\  \mbox{{\rm{and}}} \ \sigma(y) = \sigma_i^{-1}(y),\ \ y\in R_i; \\
  (2) & \Phi_i\circ \sigma(x) = \sigma'\circ\varphi_i(x), \ x\in R_i;\\
  (3) & \dfrac{ dm\circ \sigma}{dm }(z) = \rho_m(\sigma (z), \sigma_i)^{-1}.
\end{eqnarray*}
\end{lem}

\begin{proof} (1) Indeed, if $x\in D_i$, then  $y =  \sigma_i(x) \in R_i$. Hence the relation $\sigma\circ\sigma_i(x) = x$ implies $\sigma(y) = x$, that is $\sigma_i^{-1}(y) = \sigma(y)$.

(2) It follows from Definition \ref{isomorphic sfs} (c) that for any $x\in R_i$
$$
\varphi_i^{-1}\circ\sigma'_i\circ\Phi_i\circ\sigma(x) = x
$$
or $\Phi_i\circ\sigma(x) = (\sigma'_i)^{-1}\circ \varphi_i (x)$. By (1), we have
$\Phi_i\circ \sigma(x) = \sigma'\circ\varphi_i(x)$.

(3) If $z = \sigma_i(y)$, then by (1) $\sigma(z) = y$ and
\begin{eqnarray*}
  \frac{dm(\sigma(z)) }{dm(z) } &=& \frac{dm(\sigma\circ \sigma_i(y))}{dm(\sigma_i(y))}  \\
  &=&  \frac{dm(\sigma(z))}{dm(\sigma_i \circ \sigma(z))} \\
  &=&  \rho_m(\sigma(z), \sigma_i)^{-1}.
\end{eqnarray*}
\end{proof}

\begin{thm}\label{isomorphic reprs}
Let $\{\sigma_i : i \in \Lambda\}$ and $\{\sigma'_i : i \in \Lambda\}$ be two isomorphic saturated s.f.s. defined on measure spaces $(X,m)$ and $(X',m')$,  respectively. Let also  $\sigma : X\to X$ and $\sigma' : X' \to X'$ be the corresponding coding maps. Suppose that $\{T_i = T_i (m) : i\in \Lambda\}$ and $\{T'_i = T'_i(m') : i\in \Lambda\}$ are operators acting respectively on $L^2(X, m)$ and  $L^2(X', m')$ according to the formulas:
\begin{equation}\label{T_i(m) formula}
     (T_i\psi)(x) = \chi_{R_i}(x)\rho_m(\sigma (x), \sigma_i)^{-1/2} \psi(\sigma (x)), \ i\in \Lambda, \  \psi\in L^2(X,m)
\end{equation}
\begin{equation}\label{T_i(m') formula}
     (T'_i\xi)(x) = \chi_{R'_i}(x) \rho_{m'}(\sigma'(x), \sigma'_i)^{-1/2} \xi(\sigma' (x)),\  i\in \Lambda,\  \xi\in L^2(X',m')
\end{equation}
If $\wt A$ is the matrix defined by (\ref{C-K defines A}), then the representations $\pi$ and $\pi'$ of $\mathcal O_{\wt A}$ determined by $\{T_i : i\in \Lambda\}$ and $\{T'_i : i\in \Lambda\}$, respectively, are unitarily equivalent.
\end{thm}

\begin{proof}
We will show that there exists an isometry operator $U : L^2(X', m') \to L^2(X, m)$ such that
\begin{equation}\label{U intertwines}
(UT'_i\xi)(x) = (T_iU\xi)(x),\ \ \ \xi\in  L^2(X', m'), \ \ i\in \Lambda.
\end{equation}
Define $U = U_F$ by setting
\begin{equation}\label{U definition}
     (U\xi)(x) = \sqrt{\frac{dm'\circ F }{dm}}(x) \xi(F(x)), \ \ \xi\in  L^2(X', m')
\end{equation}
where $F : X \to X'$ is a map defined in Remark \ref{map F}. Firstly, we check that $U$ is an isometry. In the respective $L^2$-norms, we have:
\begin{eqnarray*}
 \Vert U\xi \Vert^2 &=& \int_X \frac{dm'\circ F }{dm }(x) \overline{\xi}(F(x)) \xi(Fx)dm(x)\\
   &=& \int_X |\xi(F(x))|^2dm'(F(x)) \\
   &=& \Vert \xi\Vert^2.
\end{eqnarray*}
We used here the fact that $F$ is a measurable one-to-one map from $X$ onto $X'$. Secondly, it is easy to notice that $U$ is onto.

Next, we check that (\ref{U intertwines}) holds. In what follows,  we will use the relation $F\circ\sigma(x) = \Phi_i\circ\sigma(x) = \sigma'\circ\varphi_i(x)$ when $x \in R_i$, and the relation of Lemma \ref{properties of sfs} (2).
\begin{eqnarray*}
   (T_iU\xi)(x) &=& \chi_{R_i}(x) \rho_{m}(\sigma(x), \sigma_i)^{-1/2}(U\xi)(\sigma (x)) \\
   &=& \chi_{R_i}(x) \rho_{m}(\sigma(x), \sigma_i)^{-1/2} \left(\frac{dm'\circ F}{dm}(x)\right)^{1/2}\xi(F\circ \sigma (x)) \\
   &=& \chi_{R_i}(x) \rho_{m}(\sigma(x), \sigma_i)^{-1/2} \left(\frac{dm'(\sigma' \circ \varphi_i(x)) }{dm(\sigma (x)) }\right)^{1/2} \xi(\sigma' \circ \varphi_i(x))\\
   &=&  \chi_{R_i}(x) \left(\frac{dm(\sigma(x))}{dm(x)}\right)^{1/2}  \left(\frac{dm'(\sigma' \circ \varphi_i(x)) }{dm(\sigma (x)) }\right)^{1/2}  \xi(\sigma' \circ \varphi_i(x))\\
& = & \chi_{R_i}(x) \left( \frac{dm'(\sigma' \circ \varphi_i(x)) } {dm(x) } \right)^{1/2} \xi(\sigma' \circ \varphi_i(x))
\end{eqnarray*}

On the other hand, we have
\begin{eqnarray*}
  (UT'_i\xi)(x) &=& \left( \frac{dm'\circ F}{dm }(x)\right)^{1/2}(T_i(m')\xi)(F(x)) \\
   &=&  \left( \frac{dm'\circ F}{dm }(x)\right)^{1/2}\chi_{R'_i}(F(x)) \rho_{m'}(\sigma'\circ F(x), \sigma'_i)^{-1/2}\xi(\sigma'\circ F(x)).
\end{eqnarray*}
We observe that $\chi_{R'_i}(F(x)) = \chi_{R'_i}(\varphi_i(x)) = \chi_{R_i}(x)$ and if $x\in R_i$, then
\begin{eqnarray*}
  \rho_{m'}(\sigma' \circ F(x), \sigma'_i)^{-1/2} &=& \left(\frac{ dm'(\sigma' \circ F(x))}{ dm'(\sigma'_i\circ \sigma'\circ F(x)) } \right)^{1/2}\\
   &=& \left( \frac{ dm'(\sigma'\circ\varphi_i(x))}{dm'(\varphi_i(x)) } \right)^{1/2}
  \end{eqnarray*}
Substituting the letter into the expression for $UT_i(m')$, we find
\begin{eqnarray*}
 (UT_i(m')\xi)(x) &=& \chi_{R_i}(x) \left( \frac{dm'\circ F}{dm }(x)
\cdot \frac{ dm'(\sigma'\circ\varphi_i(x))}{dm'(\varphi_i(x)) } \right)^{1/2} \xi(\sigma'\circ\varphi_i(x))\\
   &=& \chi_{R_i}(x) \left( \frac{ dm'(\sigma'\circ\varphi_i(x))}{dm(x) } \right)^{1/2}
\xi(\sigma'\circ\varphi_i(x)).
\end{eqnarray*}

Comparing the found expressions for $UT'_i$ and $T_iU$, we see that (\ref{U intertwines}) holds, and therefore the representations $\pi$ and $\pi'$ of $\mathcal O_{\wt A}$ are unitarily equivalent.
\end{proof}

\section{Isomorphic s.f.s. on stationary Bratteli diagrams and equivalent representations of $\mathcal O_A$}\label{section isomorphic s.f.s.}

Let $B$ and $B'$ be two stationary simple 0-1 Bratteli diagrams with incidence matrices $F$ and $F'$  whose edge sets are $E$ and $E'$ respectively. Denote as usual $A = F^T,\ A' = (F')^T$. Let now $\{\sigma_e : e \in E\}$ and $\{\sigma'_{e'} : e' \in E'\}$ be two s.f.s. defined on $(X_B, m)$ and $(X_{B'}, m')$ (see Example \ref{example of s.f.s.}) where Borel probability measures are not specified at this stage. Recall also that we denoted by $\mathcal L(E)$ ($\mathcal L(E')$) the set of linked pairs of edges of $E$ ($E'$).

Our aim is to find out  under what conditions these s.f.s. are isomorphic. We observe that the same problem can be considered on a single stationary 0-1 Bratteli diagram $B$. In this case,  the s.f.s. $\{\sigma'_e : e \in E\}$ is obtained from  $\{\sigma_e : e \in E\}$ by rearranging edges from $E$.

\begin{definition}\label{admissible map}
Let the diagrams $B$ and $B'$ be as above. Suppose that $|E| =|E'|$. Then a one-to-one map $\alpha : E \to E'$ is called \textit{admissible} if for any edges $e, f\in E$
$$
r(e) = s(f) \ \ \Longleftrightarrow \ \ r((\alpha(e)) = s(\alpha(f)),
$$
that is $\alpha\times\alpha (\mathcal L(E)) = \mathcal L(E')$.
\end{definition}

It follows from this definition that $\alpha$ is admissible if and only if $\alpha^{-1} : E' \to E$ is admissible.

\begin{example}\label{example 3x3(1)}  We give an example of an admissible map defined on the stationary 0-1 Bratteli diagram $B$ where the matrix  $A$ is taken as in Example \ref{example 3x3}. Based on Remark \ref{correspondence E=A} we establish the one-to-one correspondence between edges of   $E$ and non-zero entries of $A$ as follows: $e_1 \leftrightarrow a_{1,1}, e_2 \leftrightarrow a_{1,2}, e_3 \leftrightarrow a_{2,2}, e_4 \leftrightarrow a_{2,3}, e_5 \leftrightarrow a_{3,1}, e_6 \leftrightarrow a_{3,3}$. If we define
$$
\alpha(e_1) = e_3, \ \alpha(e_3) = e_6, \ \alpha(e_6) = e_1, \ \alpha(e_2) = e_4, \
\alpha(e_4) = e_5, \ \alpha(e_5) = e_2,
$$
then we see that $\alpha$ is an admissible map. This fact is verified directly by considering the set of linked pairs for $B$.
\end{example}

\begin{lem}\label{defn alpha bar}
Suppose that $\alpha : E \to E'$ is an admissible map where $E$ and $E'$ are the edge sets of 0-1 stationary simple Bratteli diagrams $B$ an $B'$. Then $\alpha$ generates a homeomorphism $\overline \alpha : X_B \to X_{B'}$ defined as follows: for any $ x = (x_i) \in X_B$, set  $\overline \alpha(x) = y$ where $y = (y_0, \alpha(x_1), \alpha(x_2). ... )\in X_{B'}$ and $y_0$ is uniquely determined  by $\alpha(x_1)$ (for consistency, we will denote $y_0 = \alpha(x_0)$).
\end{lem}

\begin{proof} It  easily follows from Definition \ref{admissible map} that $\overline \alpha$ is one-to-one and onto. The fact that $\overline\alpha$ is continuous is deduced from the observation that the preimage of any cylinder set is a cylinder set.
\end{proof}

Assuming that $B$ and $B'$ are 0-1 simple stationary Bratteli diagrams such that an admissible map $\alpha :E \to E'$ exists, we consider possible relations between $\overline\alpha$ and measures $m$ and $m'$ on the path spaces $X_B$ and $X_{B'}$. According to subsection \ref{measures on BD}, we will discuss three cases: the $\mathcal R$-invariant measure $\mu$, stationary Markov measures $\nu$, and general Markov measures $m$.
\medskip

(I) Let $\mu$ and $\mu'$ be the probability measures on $X_B$ and $X_{B'}$ invariant with respect to the tail equivalence relations $\mathcal R$ and $\mathcal R'$. Any such a measure is completely determined by its values on cylinder sets (see (\ref{inv_mesaure}). Let $(\mathbf{x}, \lambda)$ and $(\mathbf{x}',\lambda')$ be Perron-Frobenius data for matrices $A$ and $A'$, respectively.

Since $\overline \alpha : X_{B} \to X_{B'}$, the measure $\mu'$ defined by  $(\mathbf{x}',\lambda')$ is pulled back to $X_B$ and determines a new measure $\mu'\circ\overline\alpha$ on $X_B$.

\begin{lem}\label{alpha bar for mu}
If $\lambda \neq \lambda'$, then the measures $\mu'\circ\overline\alpha$ and $\mu$ are singular.

If $\lambda = \lambda'$,
then $\mu'\circ\overline\alpha = \mu$.
\end{lem}

\begin{proof} We recall that $\overline\alpha$ maps cylinder sets of $X_B$ onto cylinder sets of $X_{B'}$. Their measures are computed according to (\ref{inv_mesaure}): if $\overline e = (e_0, e_1, ... , e_n)$, then
$$
\mu([\overline e]) = \frac{\mathbf x_{r(e_n)}}{\lambda^{n-1}},\ \ \ \mu'(\overline\alpha([\overline e])) = \frac{\mathbf x_{r(\alpha_n(e_n))}}{(\lambda')^{n-1}}.
$$
Applying de Possel's theorem, we see that for $\{x\} = \bigcap_n[\overline e(n)]$
\begin{equation}\label{RN der alpha-mu}
\left(0 <  \frac{d\mu'\circ \overline\alpha }{d\mu }(x) = \lim_{n\to\infty} \frac{\mu'(\overline\alpha([\overline e])) } {\mu([\overline e])  }<\infty\right) \ \ \Longleftrightarrow \ \ \lambda = \lambda'.
\end{equation}

In fact the above Radon-Nikodym derivative must be equal to one. This observation can be deduced from the following argument. Note that $(x,y) \in \mathcal R$ if and only if   $(\overline\alpha (x), \overline\alpha(y)) \in \mathcal R'$. If $[\mathcal R]$ is the full group of transformations generated by $\mathcal R$, then
$$
\overline\alpha^{-1}[\mathcal R']\overline\alpha = [\mathcal R].
$$
This means that $\mu'\circ \overline\alpha$ is an $\mathcal R$-invariant ergodic measure. Since the set $\mathcal M_1(\mathcal R)$ is a singleton, we conclude that $\mu'\circ \overline\alpha = \mu$.

\end{proof}

In particular, if $B = B'$ and $\alpha$ is an admissible map from $E$ onto $E$, then the homeomorphism $\overline\alpha : X_B \to X_B$ preserves the measure $\mu$, and $\overline\alpha$ belongs to the normalizer $N[\mathcal R]$.
\medskip

(II) Consider now the case of stationary Markov measures $\nu = \nu(P)$ and $\nu' = \nu'(P')$ defined on $X_B$ and $X_{B'}$ by Markov matrices $P$ and $P'$  according to (\ref{nu(P)}).

\begin{lem}\label{alpha on stat meas}
Let $\overline\alpha : X_B \to X_{B'}$ be defined as above. Then  $\nu' \circ\overline\alpha$ is equivalent to $\nu$ if and only if for any $e\in E$
\begin{equation}\label{alpha invariance}
     p'_{s(\alpha (e)), \alpha(e)} = p_{s(e), e}.
\end{equation}
\end{lem}

\begin{proof} Let $\{x\}$ be a point that is uniquely determined by a nested sequence of cylinder sets $([\overline e(n)])$. Then $\overline \alpha ([\overline e(n)]) = [\alpha(e_0), \alpha(e_1), ... , \alpha(e_n)]$ and  we find using relation (\ref{nu(P)}) that for any $n$
\begin{equation}\label{comp RN for nu}
     \frac{\nu'(\overline \alpha ([\overline e(n)])} { \nu'[\overline e(n)] } =
\frac{(p')^{(0)}_{v_0, \alpha(e_0)} }{p^{(0)}_{v_0, e_0} } \prod_{i =1}^n\frac {p'_{s(\alpha(e_i)), \alpha(e_i)}  }{ p_{s(e_i), e_i}}
\end{equation}
If the condition of the lemma holds, then the Radon-Nikodym derivative  $\dfrac{d\nu' \circ\overline\alpha}{d\nu}(x) $ is finite and positive.

Conversely, suppose that $ \nu' \circ\overline\alpha \sim \nu$. Then
we see from (\ref{comp RN for nu}) that the product
 $$
\prod_{i =1}^\infty\frac {p'_{s(\alpha(e_i)), \alpha(e_i)}  }{ p_{s(e_i), e_i}}
$$
must converge. Indeed, this follows from the following observation: let $T$ be a finite set containing 1, then $0 <\prod_i t_i < \infty, \ t_i \in T$, if and only if $t_i =1$ for all sufficiently large $i$. Since $B$ is simple, then for any $e\in E$ and $\nu$-almost all $x =(x_i) \in X_B$ the $|\{ i : x_i =e\}| =\infty$. This proves the result.

\end{proof}

\begin{example}\label{example of P}
In order to illustrate the results proved in Lemma \ref{alpha on stat meas}, we consider the same Bratteli  diagram $B$ and matrix $A$ as in Examples \ref{example 3x3} and \ref{example 3x3(1)}. For $\alpha : E \to E$ defined in Example  \ref{example 3x3(1)}, define the matrix $P$ (that determines a stationary Markov measure on $X_B$) as follows
$$
P = \left(
  \begin{array}{cccccc}
    p & q & 0 & 0 & 0 & 0 \\
    0 & 0 & p & q & 0 & 0 \\
    0 & 0 & 0& 0 & q & p \\
  \end{array}
\right)
$$
where $p,q \in (0,1), p+q =1$. It is straightforward to verify that for chosen $\alpha$ and $P$ the condition of $\alpha$-invariance (\ref{alpha invariance}) holds.
\end{example}

(III) In the case of arbitrary Markov measures $m$ and $m'$ defined on $X_B$ and $X_{B'}$  by  sequences of stochastic matrices $(P_n)$ and $P'_n$, we can proceed as in case (II). By the same method as in Lemma \ref{alpha on stat meas}, we can show that the following statement holds (an easy proof is omitted).

\begin{lem} \label{alpha Markov meas}
Let $B$ and $B'$ be given as above and let $m = m(P_n)$ and $m' m'(P'_n)$ be Markov measures defined on $X_B$ and $X_{B'}$. Then for $\overline\alpha : X_B \to X_{B'}$, the measure $m' \circ\overline\alpha$ is equivalent to $m$ if and only if for any $x =(x_i) \in X_B$
\begin{equation}\label{alpha inv for m}
    0 < \prod_{i =1}^\infty\frac {(p')^{(i)}_{s(\alpha(x_i)), \alpha(x_i)}  }{ p^{(i)}_{s(x_i), x_i}} <\infty.
\end{equation}

\end{lem}

\begin{example}
We use the 0-1 stationary Bratteli diagram and the admissible map $\alpha$ from Example \ref{example 3x3(1)}. To define a Markov measure $m = m(P_n)$, we can set for $n\in \mathbb N$
$$
P_n = \left(
  \begin{array}{cccccc}
    p & q & 0 & 0 & 0 & 0 \\
    0 & 0 & p + \varepsilon_n & q -  \varepsilon_n & 0 & 0 \\
    0 & 0 & 0& 0 & q +\delta_n & p -\delta_n \\
  \end{array}
\right)
$$
where the sequences $\varepsilon_n$ and $\delta_n$ are chosen so that the product (\ref{alpha inv for m}) converges.
\end{example}

\begin{thm}\label{isom sfs for Br dgrms}
Let $B$ and $B'$ be two 0-1 stationary simple Bratteli diagrams and let $\{\sigma_e : e \in E\}$ and $\{\sigma'_{e'} : e' \in E'\}$ be the corresponding s.f.s. defined on $(X_B, m)$ and $(X_{B'}, m')$. Suppose that $\alpha : E \to E'$ is an admissible map and $\overline \alpha : X_B \to X_{B'}$ is the one-to-one transformation generated by $\alpha$. Then $\overline\alpha$ implements the isomorphism of $\{\sigma_e : e \in E\}$ and $\{\sigma'_{\alpha(e)} : e \in E\}$ if and only if at least one of the following conditions hold:
\begin{itemize}
\item $m = \mu$, $m' = \mu'$ where $\mu$ and $\mu'$ are the measures invariant with respect to $\mathcal R$ and $\mathcal R'$ respectively satisfying the invariance relation $\mu'\circ\overline\alpha = \mu$;

\item $m = \nu$, $m' = \nu'$ where $\nu = \nu(P)$ and $\nu' = \nu(P')$ are stationary Markov measures satisfying condition (\ref{alpha invariance});

\item $m = m(P_n)$ and $m' = m'(P'_n)$ are Markov measures satisfying condition (\ref{alpha inv for m}).
\end{itemize}
\end{thm}

\begin{proof}
The proof is based on Lemmas \ref{alpha bar for mu}, \ref{alpha on stat meas}, and \ref{alpha Markov meas} and the following  observations. According to Definition \ref{isomorphic sfs}, we need to check several properties that the s.f.s. $\{\sigma_e : e \in E\}$ and $\{\sigma'_{\alpha(e)} : e \in E\}$ must satisfy.

Firstly, we note that for any $e\in E$, one has $\overline \alpha : R_e \to R'_{\alpha(e)}$. Indeed,  $x = (x_i) \in R_e \ \Longleftrightarrow \ x_1 = e$, hence $\alpha (x_1) = \alpha(e)$, and therefore $\overline\alpha(x) = (\alpha(x_i)) \in R'_{\alpha(e)}$. Since $\alpha : E \to E'$ is one-to-one and onto, the map   $\overline \alpha : R_e \to R'_{\alpha(e)}$ is also one-to-one and onto.

Secondly, it follows from the proved fact that  $\overline \alpha(D_e) = D'_{\alpha(e)}$. Indeed,
$$
\overline \alpha(D_e) = \bigcup_{f : s(f) = r(e)} \overline \alpha(R_f) = \bigcup_{f : s(f) = r(e)}  R'_{\alpha(f)} = \bigcup_{f' : s(f') = r(\alpha(e))}  R'_{f'} = D_{\alpha(e)}.
$$

Thirdly, we claim that for any $e\in E$ and $x\in D_e$
$$
\overline\alpha\circ \sigma_e(x) = \sigma'_{\alpha(e)} \circ \overline\alpha (x).
$$
Compute
$$
\overline\alpha\circ \sigma_e(x) = \overline\alpha (x'_0, e, x_1, ... ) = (\alpha(x_0), \alpha(e), \alpha(x_1), ...)
$$
and
$$
\sigma'_{\alpha(e)} \circ \overline\alpha (x) = \sigma'_{\alpha(e)} (\alpha(x_0), \alpha(x_1), ...) = (y'_0, \alpha(e), \alpha(x_1), ...).
$$
It remains to notice that $\alpha(x_0) = y'_0$ since these edges are determined by $s(\alpha(e))$.

In general, the measure $m'\circ \overline\alpha$ is not equivalent to $m$. The proved lemmas give necessary and sufficient conditions for such an equivalence. They are the same as those given in the theorem. The proof is complete.
\end{proof}

Theorems \ref{isomorphic reprs} and \ref{isom sfs for Br dgrms} allow us to deduce  the following result.

\begin{thm}\label{equiv repr on BD}
Let $B$, $B'$, $\{\sigma_e : e \in E\}$ and $\{\sigma'_{e'} : e' \in E'\}$ be as in
Theorem \ref{isom sfs for Br dgrms}. Suppose that for an admissible map
$\alpha : E \to E'$ the transformation
$\overline \alpha : (X_B, m) \to (X_{B'}, m')$ implements an isomorphism of
s.f.s. $\{\sigma_e : e \in E\}$ and $\{\sigma'_{\alpha(e)} : e \in E\}$ (this,
in particular, means that the measures $m$ and $m'$ satisfy the conditions
of Theorem \ref{isom sfs for Br dgrms}; and the both s.f.s. have the same matrix $\wt A$ defined by (\ref{C-K defines A})). Let $\pi$ and $\pi'_{\alpha}$ be
representations of $\mathcal O_{\wt A}$ generated by operators
$\{T_e = T_e(m) : e\in E\}$ and $\{T'_{\alpha(e)} = T'_{\alpha(e)} (m') : e\in E\}$
according to the following formulas:
\begin{equation}\label{T_e formula}
  (T_e\psi)(x) = \chi_{R_e}(x)\rho_m(\sigma (x), \sigma_e)^{-1/2} \psi(\sigma (x)),
\ \psi\in L^2(X,m),
\end{equation}
\begin{equation}\label{T'_e formula}
  (T'_{\alpha(e)} \xi)(x) = \chi_{R'_{\alpha(e)}}(x) \rho_{m'}(\sigma'(x),
\sigma'_{\alpha(e)})^{-1/2} \xi(\sigma' (x)),\   \xi\in L^2(X',m')
\end{equation}
where $\sigma$ and $\sigma'$ are coding maps acting on $X_B$ and $X_{B'}$.
Then the representation $\pi$ and $\pi_\alpha$ of $\mathcal O_{\wt A}$ are unitarily  equivalent.
\end{thm}

We remark that Theorem \ref{wt A repr generates A repr} (given below) allow us to conclude that not only the representations of  $\mathcal O_{\wt A}$ are equivalent but the representations of  $\mathcal O_A$ constructed by formulas (\ref{T_i defined by wt T_e}) are also equivalent.

\begin{proof}
The proof of this theorem is absolutely similar to that of Theorem
\ref{isomorphic reprs} so that we can omit the details. We define an isometric operator $V : L^2(X_{B'}, m') \to L^2(X_B, m)$ by
$$
(V\xi)(x) = \left(\frac{dm' \circ \overline \alpha  } { dm} (x) \right)^{1/2}\xi(\overline \alpha (x)).
$$
Then we can verify by direct computations that for every $e \in E$ and $\xi \in L^2(X_{B'}, m')$
$$
(T_eV\xi)(x) = (V T_{\alpha(e)}\xi)(x).
$$
This means that the representations $\pi$ and $\pi'_\alpha$ are equivalent.
\end{proof}

As was mentioned above, the proved theorem looks simpler when $B = B'$ and $\alpha : E \to E$ is an admissible map of the set $E$.

We give below an example of two 0-1 simple stationary Bratteli diagrams whose s.f.s. are isomorphic.

\begin{example}\label{two 4x4 matrices}
Let $B$ and $B'$ be Bratteli diagrams defined by matrices
$$
A = \left(
  \begin{array}{cccc}
    1 & 0 & 1 & 1 \\
    0 & 1 & 0 & 1 \\
    0 & 1 & 0 & 0 \\
    1 & 0 & 1 & 0 \\
  \end{array}
\right), \ \ \
A' = \left(
       \begin{array}{cccc}
         1 & 0 & 1 & 1 \\
         0 & 1 & 1 & 0 \\
         1 & 0 & 0 & 1 \\
         0 & 1 & 0 & 0 \\
       \end{array}
     \right)
$$
respectively. The sets of edges $E = \{e_1, ... , e_8\}$ and $E' = \{e'_1, ... , e'_8\}$ of $B$ and $B'$ are shown on the following figures:

$
\unitlength=1cm
\begin{graph}(4,6)
 \graphnodesize{0.2}
 \roundnode{V11}(0.5,5)
 \roundnode{V12}(3,5)
 \roundnode{V13}(5.5,5)
\roundnode{V14}(8,5)
\roundnode{V21}(0.5,1.5)
 \roundnode{V22}(3,1.5)
 \roundnode{V23}(5.5,1.5)
\roundnode{V24}(8,1.5)
\graphlinewidth{0.025}
\edge{V11}{V21}
\freetext(0.2,4.3){$e_1$}
\edge{V11}{V23}
\freetext(1.2,4.2){$e_2$}
\edge{V11}{V24}
\freetext(1.8,4.8){$e_3$}
\edge{V12}{V22}
\freetext(2.7,4.7){$e_5$}
\edge{V12}{V24}
\freetext(3.8,4.8){$e_4$}
\edge{V13}{V22}
\freetext(5,4.8){$e_6$}
\edge{V14}{V21}
\freetext(7.1,4.9){$e_8$}
\edge{V14}{V23}
\freetext(8,4.5){$e_7$}
\end{graph}
\mbox{Fig. 2: Edge set $E$ of $B$.}
$

and

$
\unitlength=1cm
\begin{graph}(4,6)
 \graphnodesize{0.2}
   \roundnode{V11}(0.5,5)
  \roundnode{V12}(3,5)
 \roundnode{V13}(5.5,5)
\roundnode{V14}(8,5)
\roundnode{V21}(0.5,1.5)
 \roundnode{V22}(3,1.5)
 \roundnode{V23}(5.5,1.5)
\roundnode{V24}(8,1.5)
\graphlinewidth{0.025}
\edge{V11}{V21}
\freetext(0.2,4.3){$e'_1$}
\edge{V11}{V23}
\freetext(1,4.2){$e'_2$}
\edge{V11}{V24}
\freetext(1.8,4.8){$e'_3$}
\edge{V12}{V22}
\freetext(2.7,4.7){$e'_5$}
\edge{V12}{V23}
\freetext(3.6,4.8){$e'_4$}
\edge{V13}{V21}
\freetext(4.6, 4.8){$e'_6$}
\edge{V13}{V24}
\freetext(6.1, 4,8){$e'_7$}
\edge{V14}{V22}
\freetext(7.2,4.9){$e'_8$}
\end{graph}
\mbox{Fig. 2: Edge set $E'$ of $B'$.}
$
\\

Define an admissible map $\alpha : E \to E'$ as follows.
$$
\alpha(e_1) = e'_1, \ \alpha(e_2) = e'_3, \ \alpha(e_3) = e'_2, \ \alpha(e_4) = e'_4,
$$
$$
\alpha(e_5) = e'_5,\  \alpha(e_6) = e'_8,\  \alpha(e_7) = e'_7, \ \alpha(e_8) = e'_6.
$$
It is straightforward to check that $\alpha \times \alpha (\mathcal L(E)) = \mathcal L(E')$. Therefore, one can apply Theorem \ref{isom sfs for Br dgrms} to obtain that the corresponding s.f.s. $\{\sigma_e : e \in E\}$ and $\{\sigma'_{\overline \alpha(e)} : e \in E\}$ are isomorphic.

\end{example}

Suppose that one has a 0-1 simple stationary Bratteli diagram $B$ whose transpose to the incidence matrix is $A$. Let the other conditions of Theorem \ref{equiv repr on BD} be satisfied. Then we can construct the unitarily equivalent representations $\wt\pi$ and $\wt\pi'$ (we changed the notation) of $\mathcal O_{\wt A}$ generated respectively by the s.f.s. $\{\sigma_e : e \in E\}$ and $\{\sigma_{\alpha(e)} : e \in E\}$. In the next statements we show that these representations define simultaneously some representations $\pi$ and $\pi'$ of $\mathcal O_A$ which are also unitarily equivalent (undefined symbols are taken from Theorem \ref{equiv repr on BD}).

\begin{thm}\label{wt A repr generates A repr} Suppose that $B = B(V,E)$ is a 0-1 simple stationary Bratteli diagram, $E$ is the edge set of $B$, $V$ is the vertex set, and $A$ is the corresponding 0-1 matrix. For the s.f.s. $\{\sigma_e\ : e \in E\}$, consider the representation $\wt\pi$ of $\wt A$ generated on $L^2(X, m)$ by the operators
\begin{equation}\label{wt T_e formula}
(\wt T_e\psi)(x) = \chi_{R_e}(x)\rho_m(\sigma (x), \sigma_e)^{-1/2} \psi(\sigma (x)),
\ e \in E, \ \psi\in L^2(X,m),
\end{equation}
\begin{equation}\label{wt T^*_e formula}
  (\wt T^*_e\xi)(x) = \chi_{D_e}(x) \rho_m(x, \sigma_e)^{1/2} \xi(\sigma_e (x)),  \ e\in E , \  \xi\in L^2(X, m).
\end{equation}
For each $i\in V$, define
\begin{equation}\label{T_i defined by wt T_e}
T_i = \sum_{e : s(e) = i} \wt T_e.
\end{equation}
Then the collection of operators $\{T_i : i \in V\}$ generates a representation of $\mathcal O_A$ on $L^2(X,m)$.
\end{thm}

\begin{proof} To prove the result it suffices to check that
\begin{equation}\label{T_i satisfies CK}
     \sum_{i\in V} T_iT^*_i =1, \ \ \ \ \ T^*_iT_i = \sum_{j\in V} a_{i,j} T_jT^*_j
\end{equation}
assuming that the operators $\{\wt T_e : e \in E\}$ satisfy (\ref{T_i satisfies CK}) when the matrix $A$ is replaced by $\wt A$.

To show that the first relation in (\ref{T_i satisfies CK}) holds, we compute
\begin{eqnarray*}
    \sum_{i\in V} T_iT^*_i &=& \sum_{i\in V} (\sum_{ e : s(e) =i}\wt T_e) (\sum_{ e' : s(e') =i}\wt T^*_{e'})\\
 & = & \sum_{i\in V}(\sum_{ e : s(e) =i}\wt T_e\wt T^*_e)  +  \sum_{i\in V}(\sum_{ e \neq e' : s(e) =s(e') =i}\wt T_e \wt T^*_{e'}).
\end{eqnarray*}
Since
$$
\sum_{i\in V}(\sum_{ e : s(e) =i}\wt T_e\wt T^*_e) = \sum_{ e\in E} \wt T_e\wt T^*_e = 1,
$$
we need to verify only that $\wt T_e\wt T^*_{e'} =0$ for any edges $e \neq e'$ with $s(e) =s(e') = i$. Indeed, when we apply the formulas  (\ref{wt T_e formula}) and (\ref{wt T^*_e formula}) we obtain
$$
(\wt T_e\wt T^*_{e'} \psi)(x) = \chi_{R_e}(x)\chi_{D_{e'}}(\sigma (x)) \rho_m(\sigma (x), \sigma_e)^{-1/2}\rho_m( \sigma(x), \sigma_{e'})^{1/2} \psi (x).
$$
We claim that $\chi_{R_e}(x)\chi_{D_{e'}}(\sigma (x))  = 0 $ for all $x$. We note that $e\neq e'$ and $s(e) = s(e')$ implies $r(e) \neq r(e')$. If $x\in R_e$, then $x = (x_0, e, x_1, x_2, ...)$ and $\sigma(x) = (x'_0, x_1, x_2, ...)$ where $s(x_1) = r(e)$. On the other hand, $D_{e'} = \bigcup_{f: s(f) = r(e')} R_f$, and any point $y = (y_i) \in D_{e'}$ has the property $s(y_1) = r(e') \neq r(e)$. This proves the claim.

We will now prove that the second formula in (\ref{T_i satisfies CK}) is also true.
We have
$$
T^*_iT_i = \sum_{e : s(e) =i} \wt T^*_e \wt T_e + \sum_{e \neq e' : s(e) = s(e') = i}
\wt T^*_e \wt T_{e'} =  \sum_{e : s(e) =i} \wt T^*_e \wt T_e
$$
because the fact that $\wt T^*_e \wt T_e  = 0$ can be proved as above.
Next, since $ \wt T^*_e \wt T_e  = \sum_f  \wt a_{e,f}  \wt T_f \wt T^*_f$, we get
$$
T^*_iT_i  = \sum_{e : s(e) =i} \sum_{f : s(f) = r(e)}  \wt T_f \wt T^*_f
$$

On the other hand,
\begin{eqnarray*}
  \sum_{j \in V} a_{i,j} T_jT^*_j &=& \sum_{j \in V} a_{i,j} (\sum_{ f : s(f) =j}\wt T_f) (\sum_{ f' : s(f') =j}\wt T^*_{f'}) \\
  &=& \sum_{j \in V} a_{i,j} \sum_{ f : s(f) =j}\wt T_f\wt T^*_{f}
\end{eqnarray*}
We use now the correspondence $a_{i,j} \leftrightarrow (e, s(e) = i, r(e) = j)$, so that the latter has the form
$$
  \sum_{j \in V} a_{i,j} T_jT^*_j = \sum_{e : s(e) =i} \sum_{f : s(f) = j}  \wt T_f \wt T^*_f.
$$
This completes the proof of the proposition.
\end{proof}

The following corollary easily follows from the results proved above  and from relation (\ref{T_i defined by wt T_e}).

\begin{cor} Let $B = B(V,E)$, $A$, $\wt A$, and $\{\sigma_e : e\in A\}$ be as in Theorem \ref{wt A repr generates A repr}. For an admissible map $\alpha : E \to E$, let $\wt \pi$ and $\wt\pi_{\alpha}$ be the representations of $\mathcal O_{\wt A}$ defined by (\ref{T_e formula}) and (\ref{T'_e formula}). Then the unitary equivalence of the representations $\wt \pi$ and $\wt\pi_{\alpha}$  of $\mathcal O_{\wt A}$ implies the unitary equivalence of the corresponding representations of $\mathcal O_A$.
\end{cor}

\section{Monic representations of $\mathcal O_A$}\label{section monic}

Let $\mathcal O_A$ be the Cuntz-Krieger algebra defined by its generators $S_1, ... , S_n$ according to relations (\ref{CK relations}). Here $A = (a_{i,j})$ is a 0-1 primitive matrix. We recall (see Subsection \ref{C-K algebra}) that  $\mathcal O_A$ contains the commutative subalgebra $\mathcal D_A$ generated by $\{S_IS^*_I : I = i_1\cdots i_k,  k \in \mathbb N\}$ which is isomorphic to $C(X_A)$ (or, equivalently, $C(X_B)$ where $B$ is the stationary Bratteli diagram with matrix $A$).

We define and study a class of representations of  $\mathcal O_A$ called monic representations. This class was originally considered in \cite{DJ14a} for representations of the Cuntz algebra $\mathcal O_N$.

\begin{definition}\label{definition monic repr}
We say that a representation $\pi$ of  $\mathcal O_A$ on a Hilbert space $\mathcal H$  is {\em monic} if there is a cyclic vector $\xi \in \mathcal H$ for the abelian subalgebra $\mathcal D_A$, i.e.,
$$
\overline{\mathrm{span}} \{ T_IT_I^* : I \  \mbox{finite\ word} \} =\mathcal H
$$
where $T_I = \pi(S_I)$.
\end{definition}

\begin{example}
(1) Let $B$ be a stationary 0-1 Bratteli diagram with matrix $A$, and let $\{\sigma _e : e \in E\}$ be the s.f.s. defined by $B$ as in Example \ref{example of s.f.s.}. We {\em claim} that the representation of $\mathcal O_{\wt A}$ defined in Theorem  \ref{wt A repr generates A repr} is monic. Indeed, it is not hard to verify that $\wt T_e\wt T^*_e$ is the projection on $L^2(X_B, m)$ given by multiplication by the characteristic function $\chi_{R_e},\ e \in E$. Then, for any finite path $w = (e_1, ... , e_k) \in E^*$, the projection $\wt T_w\wt T^*_w$ corresponds to the multiplication by the characteristic function $\chi_{[\overline w]}$ of the cylinder set $[\overline w] = [(e'_0, e_1, ... , e_k)]$. Since $C(X_B)$ is dense in $L^2(X_B, m)$, we see that $\{\wt T_e : e \in E\}$ satisfies Definition \ref{definition monic repr}.

It follows from Lemma \ref{isom comm algs} that the representation of $\mathcal O_A$ defined as in Theorem \ref{wt A repr generates A repr}  is also monic because the projection $T_iT_i^*$ corresponds to the multiplication by the characteristic function $\chi_{[i]}$ where $[i]$ denotes all paths in $X_A$ beginning at $i$.

(2) The definition of s.f.s. based on a topological Markov chain $X_A$ (see Example \ref{ex s.f.s. topol Markov chain}) gives us another example of a monic representation. Indeed, we again can use the formulas from Theorem\ref{wt A repr generates A repr} and show that the constant function that equals 1 is cyclic for the representation of $C(X_A)$.
\end{example}

\begin{definition}\label{def monic system} Let  $\{\sigma_i : i \in \Lambda\}$ be a s.f.s. on a probability measure space $(X,m)$ with $\sigma_i : D_i \to R_i$ and a coding map $\sigma$.  We say that a collection $(m, (f_i)_{i\in \Lambda})$ is a {\em monic system} if for any $i\in \Lambda$ one has $m\circ \sigma^{-1}_i << m$ and  \begin{equation}\label{def monic system derivative}
     \rho_m(x, \sigma_i^{-1}) = |f_i|^2
\end{equation}
for some functions $f_i \in L^2(X,m)$ such that $f_i \neq 0$ for $m$-a.e. $x \in R_i$. A monic system is called {\em nonnegative} if $f_i \geq 0$.
\end{definition}

\begin{remark}
Relation (\ref{def monic system derivative}) can be written also as follows
$$
\rho_m(\sigma x, \sigma_i)^{-1} = |f_i|^2.
$$
To see this, we use the formulas $\rho_m(\sigma^{-1}_i x, \sigma_i) \rho_m(x, \sigma_i^{-1}) = 1$ and $\sigma_i^{-1} x = \sigma x, x\in R_i$. This formula agrees with that used in the definition of $T_i$ and $T^*_i$ (see (\ref{T_i formula}) and (\ref{T^*_i formula})).
\end{remark}

\setcounter{footnote}{0}

\begin{lem}\label{repr gener monic system} Suppose that a saturated s.f.s. $\{\sigma_i : i \in \Lambda\}, \sigma_i : D_i \to R_i,$ defined on $(X,m)$ satisfies condition (C-K) from Definition \ref{s.f.s.}; and let $\sigma$ be a coding map for $\{\sigma_i : i \in \Lambda\}$. Then, for any monic system $(m, (f_i)_{i\in \Lambda})$, the operators
\begin{equation}\label{repr T_i by monic system}
     T_i f = f_i \cdot f\circ\sigma, \  \ f\in L^2(x,m), \  i\in \Lambda
\end{equation}
generates a representation of $\mathcal O_{A}$ on $L^2(X,m)$  where $A$ is defined in (\ref{C-K defines A})\footnote{Here we changed the notation  and used $A$ instead of $\wt A$.}.
\end{lem}

The representation of $\mathcal O_{A}$ defined by $\{T_i : i\in \Lambda\}$ as in Lemma \ref{repr gener monic system} will be called {\em associated to a monic system}.

\begin{proof} This lemma can be proved in the same manner as \cite[Theorem 2.7]{DJ14a}. We first notice that $f_i(x) \neq 0 $ if and only if $x \in R_i = \sigma_i(D_i)$. Since $\sigma\circ \sigma_i =1$,  we observe that $T_i$ is a partial isometry:
$$
||T_if||^2 = \int_X |f_i|^2|f\circ \sigma|^2\; dm = \int_X |f\circ \sigma|^2 \;  d(m\sigma_i^{-1}) =  \int_X |f|^2 \; dm.
$$
Next, if $i \neq j$, then
$$
<T_i f, T_j g> = \int_X \overline f_i f_j \overline{ f\circ \sigma} (g\circ\sigma)\; dm = 0
$$
because $f_i$ and $f_j$ are supported by disjoint sets.

To find $T^*_i$, we define for each $i \in \Lambda$
\begin{equation}\label{def of g_i}
    g_i(x) = \frac{f_i(x)}{|f_i(x)|^2} \ \ \mbox{if \ $x\in R_i$},\ \ g_i(x) = 0 \ \ \mbox{if \ $x\notin R_i$}.
\end{equation}
Then
$$
<T^*_if, g> = \int_X \overline f T_ig \; dm = \int_X \overline f (g\circ \sigma) g_i |f_i|^2\; dm = \int_X \overline {(f\circ\sigma_i)} (g\circ \sigma_i)g \; dm,
$$
so that
$$
T^*_if = \overline {(g\circ \sigma_i)} (f\sigma_i).
$$
It follows from the proved formulas that $T_iT^*_i$ is the projection on $L^2(X,m)$ given by
\begin{equation}\label{T_iT^*_i is proj}
     T_iT^*_i f = f_i \overline{(g_i\circ \sigma_i\circ \sigma)} (f\circ \sigma_i\circ \sigma) = \chi_{R_i} f
\end{equation}
(we use here  (\ref{def of g_i}) and the relation $\sigma_i\circ \sigma(x) = x$ for $x\in R_i$).  Since $\bigcup_{i\in \Lambda} R_i = X$ and this union is disjoint, we have
$$
\sum_{i\in \Lambda} T_iT^*_i =1.
$$

To finish the proof, we check that for any $i\in \Lambda$
\begin{equation}\label{condition for T^*_iT_i}
     T^*_iT_i = \sum_{j\in \Lambda} a_{i,j} T_jT^*_j.
\end{equation}
One can show similarly to (\ref{T_iT^*_i is proj}) that for any $f \in L^2(X,m)$
$$
T^*_iT_i f = \chi_{D_i} f.
$$
It is clear that
$$
\sum_{j\in \Lambda} a_{i,j} T_jT^*_j f = \sum_{j\in \Lambda_i}  T_jT^*_j f= \sum_{j\in \Lambda_i} \chi_{R_j} f =  \chi_{D_i} f
$$
and therefore (\ref{condition for T^*_iT_i}) is proved.

\end{proof}

{\em Question:} Under what conditions on a s.f.s. $\{\sigma_i : i \in \Lambda\}$ the representation $T_i f = f_i \cdot f\circ\sigma$ is monic?
\\

We now consider a class of monic systems that is naturally related to $\mathcal O_A$. Let $X_A$ be the topological Markov chain, and let $\Sigma = \{\sigma_i : i =1, ... , n\}$ be the s.f.s. defined in Example \ref{ex s.f.s. topol Markov chain}. A monic system $(m, (f_i))$ on $X_A$ defined by $\Sigma$ is called {\em inherent}.

\begin{thm}\label{thm on inherent systems}
Let $\pi$ be a representation of $\mathcal O_A$ defined by a monic system according to (\ref{repr T_i by monic system}). Then $\pi$ is monic if and only if it is unitarily equivalent to a representation associated to an inherently  monic system.
\end{thm}

\begin{proof}
The fact that $\pi$ is a representation of $\mathcal O_A$ is proved in Lemma \ref{repr gener monic system}. Suppose that $\pi$ is generated by $\{T_i : i =1, ... , n\}$ and acts on $L^2(X_A, \mu)$. Assume that $\pi$ is associated to an inherent  monic system $\Sigma = \{\sigma_i : i =1, ... , n\}$. By direct computation we can show that the projection $T_IT^*_I$ is a multiplication by the cylinder set $\chi_{[I]}(x)$ where $I$ is any finite word on the alphabet $(1, ... , n)$. It suffices to note that the function $f \equiv 1$ is cyclic for $D_A$.

To prove the converse, we will follow \cite{DJ14a}. If $\pi$ is monic on $\mathcal H$, then there exists a cyclic vector $\xi \in \mathcal H$ for $\mathcal D_A$. It follows from the isomorphism of $\mathcal D_A$ and $C(X_A)$ that there exists a Borel measure $\mu$ on $X_A$ such that
$$
<\xi, \pi(f)\xi> \ = \int_{X_A}\: f\; d\mu,\ \ f \in C(X_A).
$$
Moreover, $\pi(\mathcal D_A)$ is unitarily equivalent to the representation of $C(X_A)$ acting on $L^2(X,\mu)$ by multiplication operators.

Consider the map $W : C(X_A) \to \mathcal H$ defined by the relation $W(f) = \pi(f) \xi$. Clearly, $W$ is linear and isometric, so that $W$ can be extended to an isometry from $L^2(X_A, \mu)$ to $\mathcal H$. We notice that this isometry is onto because $\pi$ is monic.

We now define the operators $\widehat T_i := W^*T_iW \ (i =1, ... , n)$ acting on $L^2(X_A, \mu)$. We check that $\widehat T_i$ satisfy relation (\ref{repr T_i by monic system}). Set up $f_i = \widehat T_i 1$. In what follows we will use the relations
$$
T_i^* \pi(f) T_i = \pi(f\circ\sigma_i), \ \ \  T_i\pi(f)  = \pi(f\circ\sigma)T_i
$$
which can be verified first on characteristic functions of cylinder sets.

We have
$$
\int_{X_A}|f_i|^2 f \; d\mu = < \wt T_i, f \wt T_i>_{L^2(X,\mu)} = <T_i \xi,
\pi(f) T_i\xi>_{\mathcal H}
$$
$$
= <\xi , \pi(f\circ \sigma_i)\xi>_{\mathcal H} = \int_{X_A} f\circ \sigma_i \; d\mu.
$$
Hence, we deduce that $\mu\circ \sigma_i^{-1} << \mu$ and $\dfrac{d\mu\circ \sigma_i^{-1}}{d\mu} = |f_i|^2$.

Next, if $f\in C(X_A)$, then
$$
\wt T_i f = W^*T_i \pi(f)\xi = W^*\pi(f\circ\sigma)T_i\xi = W^*\pi(f\circ\sigma)WW^*T_iW1 = (f\circ\sigma)f_i.
$$
This proves (\ref{repr T_i by monic system}). The relation $\sum_i T_iT^*_i = 1$ implies that the support of $f_i$ is the set $R_i= \{ (x_k) \in X_A : x_0 = i\}$.
\end{proof}

\begin{cor}
Let $(m, (f_i))$ and $(m', (f'_i))$ be two inherent monic systems on $X_A$. Then the representations of $\mathcal O_A$ associated to these monic systems are  equivalent if and only if the measures $m$ and $m'$ are equivalent, and there exists a function $h$ on $X_A$ such that
$$
\frac{dm'}{dm}(x) = |h(x)|^2, \ \ \ f'_i = h\circ \sigma f_i h^{-1},\  i =1, ... , n.
$$
\end{cor}

The {\em proof} is the same as that of \cite[Theorem 2.9]{DJ14a}.

\bibliographystyle{plain}

\end{document}